\RequirePackage{fix-cm}
\documentclass[smallextended]{svjour3}       
\smartqed  
\usepackage{graphicx}
\usepackage{enumerate}
\usepackage{color}
\usepackage{amsfonts,amssymb} 
\usepackage{amsmath} 
\begin{document}

\title{Numerical Approximation for Stochastic Nonlinear Fractional Diffusion Equation Driven by Rough Noise
	\thanks{This work was supported by National Natural Science Foundation of China under Grant No. 12071195, AI and Big Data Funds under Grant No. 2019620005000775, and Fundamental Research Funds for the Central Universities under Grant Nos. lzujbky-2021-it26 and lzujbky-2021-kb15.
	}
}

\titlerunning{Numerical Approximation for Stochastic Fractional Diffusion Equation}        

\author{Daxin Nie         \and
		Jing Sun\and
        Weihua Deng 
}


\institute{Daxin Nie \at
              School of Mathematics and Statistics, Gansu Key Laboratory of Applied Mathematics and Complex Systems, Lanzhou University, Lanzhou 730000, P.R. China \\
              \email{ndx1993@163.com}           
           \and
           Jing Sun \at
           School of Mathematics and Statistics, Gansu Key Laboratory of Applied Mathematics and Complex Systems, Lanzhou University, Lanzhou 730000, P.R. China \\
           \email{js@lzu.edu.cn}           
           \and
           Weihua Deng \at
              School of Mathematics and Statistics, Gansu Key Laboratory of Applied Mathematics and Complex Systems, Lanzhou University, Lanzhou 730000, P.R. China\\
              \email{dengwh@lzu.edu.cn}
}

\date{Received: date / Accepted: date}

\maketitle

\begin{abstract}
In this work, we are interested in building the fully discrete scheme for stochastic fractional diffusion equation driven by fractional Brownian sheet which is temporally and spatially fractional with Hurst parameters $H_{1}, H_{2} \in(0,\frac{1}{2}]$. We first provide the regularity of the solution. Then we employ the Wong-Zakai approximation to regularize the rough noise and discuss the convergence of the approximation. Next, the finite element and backward Euler convolution quadrature methods are used to discretize  spatial and temporal operators for the obtained regularized equation, and the detailed error analyses are developed.  Finally, some numerical examples are presented to confirm the theory.
\end{abstract}

\section{Introduction}
The Brownian motion subordinated by inverse $\alpha$-stable L\'evy process is a powerful model for describing the subdiffusion phenomena \cite{Deng.2020FDE}. 
In this paper, we are concerned with the Fokker-Planck equation (governing the probability density function of the subordinated Brownian motion) with nonlinear source term and external noises, i.e., we present and analyze the fully discrete scheme for the following stochastic nonlinear fractional diffusion equation driven by fractional Brownian sheet noise
\begin{equation}\label{eqretosol}
	\left\{\begin{aligned}
		&\partial_{t}u(x,t)+{}_{0}\partial_{t}^{1-\alpha}A u(x,t)=f(u)+\beta\xi^{H_{1},H_{2}}(x,t)\qquad (x,t)\in{D}\times (0,T],\\
		&u(x,0)=0\qquad x\in{D},\\
		&u(x,t)=0\qquad (x,t)\in\partial{D}\times (0,T],
	\end{aligned}\right.
\end{equation}
where $A=-\Delta$ with a zero Dirichlet boundary condition; ${}_{0}\partial^{1-\alpha}_{t}$ with $\alpha\in(0,1)$ is the Riemann-Liouville fractional derivative defined by \cite{Podlubny.1999FDE}
\begin{equation*}
	{}_{0}\partial^{1-\alpha}_{t}u=\frac{1}{\Gamma(\alpha)}\frac{\partial}{\partial t}\int_{0}^{t}(t-\xi)^{\alpha-1}u(\xi)d\xi;
\end{equation*}
$f(u)$ is a nonlinear term satisfying the following assumptions
\begin{equation}\label{eqnonassump}
	\begin{aligned}
		&\|f(u)\|_{L^{2}(D)}\leq C(1+\|u\|_{L^{2}(D)}),\\
		&\|f(u)-f(v)\|_{L^{2}(D)}\leq C\|u-v\|_{L^{2}(D)}
	\end{aligned}
\end{equation}
with $C$ being a positive constant; $\beta$ is a non-zero constant, and without loss of generality, we take $\beta=1$ in our analyses; $\xi^{H_{1},H_{2}}$ is defined by
\begin{equation}\label{eqdefxi}
	\xi^{H_{1},H_{2}}(x,t)=\frac{\partial^{2}W^{H_{1},H_{2}}(x,t)}{\partial x\partial t}
\end{equation}
with $W^{H_{1},H_{2}}(x,t)$ being a fractional Brownian sheet on a stochastic basis $(\Omega,\mathcal{F},(\mathcal{F}_{t})_{t\in[0,T]},\mathbb{P})$ such that
\begin{equation*}
	\begin{aligned}
		&\mathbb{E}\left[W^{H_{1},H_{2}}(x,t)W^{H_{1},H_{2}}(y,s)\right]
\\
		&
=\frac{x^{2H_{1}}+y^{2H_{1}}-|x-y|^{2H_{1}}}{2}\times \frac{t^{2H_{2}}+s^{2H_{2}}-|t-s|^{2H_{2}}}{2},\\
		& {\rm where}~ (x,t),(y,s)\in {D}\times[0,T].
	\end{aligned}
\end{equation*}
Here $D=(0,l)$ with $l$ being a bounded constant  and  $H_{1},~H_{2}\in(0,\frac{1}{2}]$ are Hurst parameters.

As we all know, fractional noises exist  widely in the natural world, such as flows in porous media, the rough Hamiltonian systems, water flows in hydrology and so on \cite{Biagini.2008SCfFBMaA,Comegna.2013UoafBmmtmshvospahpdapv}. In recent years, there have been many discussions about numerically solving stochastic partial differential equations driven by fractional Brownian motion with Hurst parameter $H\in(\frac{1}{2},1)$ (it can be called as ``smoother noise'') \cite{Arezoomandan.2021ScmfspdewfBm,Nie.2020ScoftsofdedbfGn,Wang.2017SmrrfSwfnaocrftna,Yan.2019OeeffspdewfBm}. But for the case $H\in(0,\frac{1}{2})$ (called as ``rough noise''), the existing discussions seem to be few. In \cite{Cao.2018FeafssdedbfBm}, the authors propose numerical analyses about the second-order stochastic differential equation driven by spatial fractional Gaussian noise with $H\in(0,\frac{1}{2})$; the reference \cite{Nie.2021AucaftfdedbfGnwHi$} uses the equivalence of different fractional Sobolev spaces and the assumption $\tau<\tau^{*}$ ($\tau^{*}$ depends on the spatial discretization) to provide a unified strong convergence analysis for fractional stochastic partial differential equation driven by fractional cylinder noise with $H\in(0,1)$; in \cite{Cao.2017Aseewawarn}, the authors propose the regularity estimates and the corresponding numerical analyses about the stochastic evolution equation  driven by fractional Brownian sheet with $H_{1}\in(0,\frac{1}{2})$ and $H_{2}=\frac{1}{2}$.

In this paper, we focus on the fractional diffusion equation driven by fractional Brownian sheet with Hurst parameters $H_{1},H_{2}\in(0,\frac{1}{2}]$ numerically, which are both rough in the temporal and spatial directions. Firstly, with the help of the obtained new estimate about stochastic integral with respect to $\xi^{H_{1},H_{2}}$ (for the details, see Theorem \ref{thmisometry}), we provide the regularity estimate of the solution, i.e.,
\begin{equation*}
	\begin{aligned}
		\mathbb{E}\|A^{\sigma}u(t)\|_{L^{2}(D)}^{2}\leq C,\quad 2\sigma\in \left[0,\min\left\{\frac{2H_{2}}{\alpha}+H_{1}-1,2H_{1}+1\right\}\right),
	\end{aligned}
\end{equation*}
and
\begin{equation*}
	\mathbb{E}\left \|\frac{u(t)-u(t-\tau)}{\tau^{\gamma}}\right \|_{L^{2}(D)}^{2}\leq C,\qquad \gamma\in[0,2H_{2}+(H_{1}-1)\alpha).
\end{equation*}
Then the Wong-Zakai approximation \cite{EugeneWong.1965OtCoOItSI,Eugene.1965Otrboasde} is used to regularize the fractional Brownian sheet noise $\xi^{H_{1},H_{2}}$. The existing discussions \cite{Cao.2017Aseewawarn,Cao.2018FeafssdedbfBm,Liu.2019WaosAe} rely on the Green function of Eq. \eqref{eqretosol} composed of Mittag-Leffler function \cite{Sakamoto.2011Ivvpffdeaatsip}, making the convergence analysis of the Wong-Zakai approximation complicated. So, in this paper, a new approach based on approximation theory, operator theory, and the equivalence of different Sobolev spaces is built and an $\mathcal{O}(h^{\frac{4H_{2}}{\alpha}+2H_{1}-2-\epsilon}+\tau^{2H_{2}-\frac{\alpha}{2}-\epsilon}h^{2H_{1}-1})$ convergence rate of Wong-Zakai approximation is obtained. Next, we use the finite element method and backward Euler convolution quadrature method to build the fully discrete scheme of Eq. \eqref{eqretosol} and introduce some new techniques to obtain an $\mathcal{O}(\tau^{2H_{2}+(H_{1}-1)\alpha-\epsilon})$ convergence rates in time  without assuming $\tau<\tau^{*}$.

The rest of this paper is organized as follows. In Section \ref{sec2}, we first provide some properties about stochastic integral with respect to $\xi^{H_{1},H_{2}}$, and then discuss the regularity of the solution. Next, we consider the Wong-Zakai approximation of Eq. \eqref{eqretosol} and discuss its convergence in Section \ref{sec3}. In Section \ref{sec4}, we construct the numerical scheme by finite element method and backward Euler convolution quadrature method, and provide the complete error estimates. In Section \ref{sec5}, a variety of numerical experiments are provided to verify the proposed theoretical results. At last, we conclude the paper with some discussions. Throughout the paper, $C$ denotes a positive constant, whose value may vary from line to line, $\|\cdot\|$ denotes the operator norms from $L^{2}(D)$ to $L^{2}(D)$,  $\epsilon,\epsilon_{0}>0$ are arbitrarily small quantities, and $\mathbb{E}$ denotes the expectation.
\section{Preliminaries}\label{sec2}
In this section, we begin by discussing the properties of stochastic integrals with respect to fractional Brownian sheet noise. Also, the regularity of the solution is provided.

\subsection{Some properties of stochastic integrals with respect to  fractional Brownian sheet noise}
Let $\{\lambda_{k}\}_{k=1}^{\infty}$ and $\{\phi_{k}\}_{k=1}^{\infty}$ be the non-decreasing eigenvalues and $L_{2}$-norm normalized eigenfunctions of operator $A$ with a zero Dirichlet boundary condition, respectively.  According to \cite{Laptev.1997DaNepodiEs,Li.1983OtSeatep}, the eigenvalues  $\{\lambda_{k}\}_{k=1}^{\infty}$ satisfy the following lemma.
\begin{lemma}[\cite{Laptev.1997DaNepodiEs,Li.1983OtSeatep}]\label{thmeigenvalue}
	Let $D$ be a bounded domain in $\mathbb{R}^d \,(d=1,2,3)$ and $\lambda_k$ the $k$-th eigenvalue of the Dirichlet boundary problem for the Laplace operator $A=-\Delta$ in $D$. Then, for all $k\geq 1$,
	\begin{equation*}
		\lambda_{k} \geq \frac{C_{d} d}{d+2} k^{2 / d}|D|^{-2 / d},
	\end{equation*}
	where $C_{d}=(2 \pi)^{2} B_{d}^{-2 / d}$, $|D|$ is the volume of $D$, and $B_d$ means the volume of the unit $d$-dimensional ball.
\end{lemma}
Then we present some fractional Sobolev spaces, which can refer to \cite{Acosta.2019Feaffep,Acosta.2017AfLeRosafea,Bonito.2019NaotifL,DiNezza.2012HgttfSs}.
Introduce the operator $A^{q}$ with $q\in[0,1]$ as
\begin{equation*}
	A^{q}u=\sum_{k=1}^{\infty}\lambda_{k}^{q/2}(u,\phi_{k})\phi_{k}
\end{equation*}
and define $\hat{H}^{2q}(D)=\mathbb{D}(A^{q})$ with norm $\|u\|_{\hat{H}^{2q}(D)}=\|A^{q}u\|_{L^{2}(D)}$. Here $\mathbb{D}(A^{q})$ denotes the domain of $A^{q}$. It is easy to verify
\begin{equation*}
	\hat{H}^{0}(D)=L^{2}(D),\quad \hat{H}^{1}(D)=H^{1}_{0}(D).	
\end{equation*}
For $s\in(0,1)$, we define the fractional Sobolev space $H^{s}(D)$ by
\begin{equation*}
	H^{s}(D)=\left\{u\in L^{2}(D):|u|_{H^{s}(D)}^{2}=\int_{D}\int_{D}\frac{(u(x)-u(y))^{2}}{|x-y|^{1+2s}}dxdy<\infty\right\}, 
\end{equation*}
and its norm can be written as $\|\cdot\|_{H^{s}(D)}=\|\cdot\|_{L^{2}(D)}+|\cdot|_{H^{s}(D)}$.
\begin{remark}\label{Remeq0}
	According to \cite{Ervin.2006Vfftsfade}, the semi-norm of $H^{s}(D)$ can also be defined by
	\begin{equation*}
		|u|_{H^{s}(D)}=\|{}_{0}\partial^{s}_{x}u\|_{L^{2}(D)},
	\end{equation*}
	and combining the fractional Poinc\'are inequality \cite{DiNezza.2012HgttfSs,Ervin.2006Vfftsfade}, we can also define the norm of $H^{s}(D)$ by
	\begin{equation*}
		\|u\|_{H^{s}(D)}=\|{}_{0}\partial^{s}_{x}u\|_{L^{2}(D)},
	\end{equation*}
	where ${}_{0}\partial^{s}_{x}u$ is the Riemann-Liouville fractional derivative.
\end{remark}

Moreover, for $s\in(0,1)$, another type of fractional Sobolev space \cite{Acosta.2019Feaffep,Acosta.2017AfLeRosafea,Bonito.2019NaotifL} that we will use can be defined by
\begin{equation*}
	H^{s}_{0}(D)=\{u\in H^{s}(\mathbb{R}),~~u=0~~{\rm in}~~D^{c}\}
\end{equation*}
with the norm
\begin{equation*}
	\begin{aligned}
		\|u\|_{{H}^{s}_{0}(D)}^{2}=&\|u\|_{L^{2}(D)}^{2}+|u|_{{H}^{s}_{0}(D)}^{2}\\
		=&\|u\|_{L^{2}(D)}^{2}+\int_{\mathbb{R}}\int_{\mathbb{R}}\frac{(u(x)-u(y))^{2}}{|x-y|^{1+2s}}dxdy.
	\end{aligned}
\end{equation*}
\begin{remark}\label{Remeq}
	It is well-known that $H^{s}(D)={H}^{s}_{0}(D)$ for $s\in[0,\frac{1}{2})$; see \cite{Acosta.2019Feaffep,Bonito.2019NaotifL}. From \cite{Bonito.2019NaotifL}, we have $\hat{H}^{s}(D)={H}^{s}_{0}(D)$ for $s\in[0,\frac{3}{2})$.
\end{remark}
For fractional Brownian sheet noise, we have the following It\^{o} isometry
\begin{lemma}[\cite{Bardina.2006MfiwHplt$}]\label{thmisometry0}
	Let $g_{1}(x,t)=g_{1,1}(x)g_{1,2}(t)$ and $g_{2}(x,t)=g_{2,1}(x)g_{2,2}(t)$ satisfying $g_{1,1}(x),g_{2,1}(x)\in H^{\frac{1-2H_{1}}{2}}_{0}(D)$ and $g_{1,2}(t),g_{2,2}(t)\in H^{\frac{1-2H_{2}}{2}}_{0}((0,T))$. Then we have
	\begin{equation*}
		\begin{aligned}
			&\mathbb{E}\left (\int_{0}^{T}\int_{{D}} g_{1}(x,t)\xi^{H_{1},H_{2}}(dx,dt)\int_{0}^{T}\int_{{D}} g_{2}(x,t)\xi^{H_{1},H_{2}}(dx,dt)\right )\\
			&\qquad\qquad\qquad\qquad\qquad\qquad= (\mathcal{L}_{H_{2},t}g_{1,2}(t),g_{2,2}(t))_{\mathbb{R}}(\mathcal{L}_{H_{1},x}g_{1,1}(x),g_{2,1}(x))_{\mathbb{R}},
		\end{aligned}
	\end{equation*}
	where
	\begin{equation*}
		\mathcal{L}_{H_{1},x}u(x)=\left \{\begin{aligned}
			&2C_{H_{1}}\int_{\mathbb{R}}\frac{u(x)-u(y)}{|x-y|^{2-2H_{1}}}dy \qquad H_{1}\in \left(0,\frac{1}{2} \right),\\
			&u(x) \qquad\qquad\qquad\qquad\qquad\, H_{1}=\frac{1}{2},
		\end{aligned}\right .
	\end{equation*}
	and
	\begin{equation*}
		\mathcal{L}_{H_{2},t}u(t)=\left \{\begin{aligned}
			&2C_{H_{2}}\int_{\mathbb{R}}\frac{u(t)-u(r)}{|t-r|^{2-2H_{2}}}dr \qquad H_{2}\in(0,\frac{1}{2}),\\
			&u(t) \qquad\qquad\qquad\qquad\qquad\, H_{2}=\frac{1}{2}.
		\end{aligned}\right .
	\end{equation*}
	Here $C_{H_{i}}=\frac{1}{2}H_{i}(1-2H_{i})$, $i=1,2$.
\end{lemma}
Moreover, we can obtain
\begin{theorem}\label{thmisometry}
	Let $g_{1}(x,t)=g_{1,1}(x)g_{1,2}(t)$ and $g_{2}(x,t)=g_{2,1}(x)g_{2,2}(t)$ satisfying $g_{1,1}(x),g_{2,1}(x)\in H^{\frac{1-2H_{1}}{2}}_{0}(D)$ and $g_{1,2}(t),g_{2,2}(t)\in H^{\frac{1-2H_{2}}{2}}_{0}((0,T))$. Then we have
	\begin{equation*}
		\begin{aligned}
			&\mathbb{E}\left (\int_{0}^{T}\int_{{D}} g_{1}(x,t)\xi^{H_{1},H_{2}}(dx,dt)\int_{0}^{T}\int_{{D}} g_{2}(x,t)\xi^{H_{1},H_{2}}(dx,dt)\right )\\
			&\qquad\qquad\leq C\left \|{}_{0}\partial^{\frac{1-2H_{2}}{2}}_{t}g_{1,2}(t)\right \|_{L^{2}((0,T))}\left \|{}_{0}\partial^{\frac{1-2H_{2}}{2}}_{t}g_{2,2}(t)\right \|_{L^{2}((0,T))}\\
			&\qquad\qquad\qquad\cdot\|g_{1,1}(x)\|_{H^{\frac{1-2H_{1}}{2}}_{0}(D)}\|g_{2,1}(x)\|_{H^{\frac{1-2H_{1}}{2}}_{0}(D)}.
		\end{aligned}
	\end{equation*}	
Here ${}_{0}\partial^{\alpha}_{t}$ is the Riemann-Liouville fractional derivative when $\alpha\in(0,1)$; and when $\alpha=0$, it denotes an identity operator.
\end{theorem}
\begin{proof}
	Here, we mainly prove the case that $H_{1},H_{2}\in(0,\frac{1}{2})$. As for the other cases, the desired results can be obtained similarly.
	
According to Lemma \ref{thmisometry0} and the definition of $H^{s}_{0}(D)$, after simple calculations, we have
\begin{equation*}
		\begin{aligned}
			&\mathbb{E}\left (\int_{0}^{T}\int_{{D}} g_{1}(x,t)\xi^{H_{1},H_{2}}(dx,dt)\int_{0}^{T}\int_{{D}} g_{2}(x,t)\xi^{H_{1},H_{2}}(dx,dt)\right )\\
			=&C_{H_{1}}C_{H_{2}}\int_{\mathbb{R}}\int_{\mathbb{R}}\frac{(g_{1,1}(x)-g_{1,1}(y))(g_{2,1}(x)-g_{2,1}(y))}{|x-y|^{2-2H_{1}}}dxdy\\
			&\quad \cdot\int_{\mathbb{R}}\int_{\mathbb{R}}\frac{(g_{1,2}(t)-g_{1,2}(s))(g_{2,2}(t)-g_{2,2}(s))}{|t-s|^{2-2H_{2}}}dtds\\
			\leq&C\|g_{1,2}(t)\|_{H^{\frac{1-2H_{2}}{2}}_{0}((0,T))}\|g_{2,2}(t)\|_{H^{\frac{1-2H_{2}}{2}}_{0}((0,T))}\\
			&\quad\cdot\|g_{1,1}(x)\|_{H^{\frac{1-2H_{1}}{2}}_{0}(D)}\|g_{2,1}(x)\|_{H^{\frac{1-2H_{1}}{2}}_{0}(D)}.
		\end{aligned}
	\end{equation*}
	Combining further Remarks \ref{Remeq0} and \ref{Remeq}, one can arrive at the desired result.
\end{proof}

\subsection{Regularity of the solution}
Before building the regularity of the solution of Eq. \eqref{eqretosol}, we first give the presentation of the solution. Introduce $\mathcal{G}(t,x,y)$ as
\begin{equation}\label{eqdefG}
	\mathcal{G}(t,x,y)=\sum_{k=1}^{\infty}\mathcal{G}_{k}(t,x,y),
\end{equation}
where
\begin{equation}\label{eqdefGk}
	\mathcal{G}_{k}(t,x,y)=E_{k}(t)\phi_{k}(x)\phi_{k}(y)
\end{equation}
and
\begin{equation}\label{eqdefEk}
	E_{k}(t)=\frac{1}{2\pi\mathbf{i}}\int_{\Gamma_{\theta,\kappa}}e^{zt}z^{\alpha-1}(z^{\alpha}+\lambda_{k})^{-1}dz.
\end{equation}
Here $\Gamma_{\theta,\kappa}$ is defined by
\begin{equation*}
	\Gamma_{\theta,\kappa}=\{r e^{-\mathbf{i}\theta}: r\geq \kappa\}\cup\{\kappa e^{\mathbf{i}\psi}: |\psi|\leq \theta\}\cup\{r e^{\mathbf{i}\theta}: r\geq \kappa\},
\end{equation*}
where the circular arc is oriented counterclockwise and the two rays are oriented with an increasing imaginary part and $\mathbf{i}^2=-1$.

Thus the solution of Eq. \eqref{eqretosol} can be written as
\begin{equation}\label{eqsolpresentation0}
	u(x,t)=\int_{0}^{t}\int_{D}\mathcal{G}(t-s,x,y)f(u)dsdy+\int_{0}^{t}\int_{D}\mathcal{G}(t-s,x,y)\xi^{H_{1},H_{2}}(dy,ds).
\end{equation}
For the convenience of analyzing, we introduce the operator $\mathcal{R}$, which is defined by the Laplace transform, i.e.,
\begin{equation*}
	\tilde{\mathcal{R}}(z)=z^{\alpha-1}(z^{\alpha}+A)^{-1}.
\end{equation*}
It is easy to verify
\begin{equation*}
	\mathcal{G}_{k}(t,x,y)=\mathcal{R}(t)\phi_{k}(x)\phi_{k}(y)
\end{equation*}
and
\begin{equation*}
	\mathcal{R}(t)u(x)=\int_{D}\mathcal{G}(t,x,y)u(y)dy.
\end{equation*}
So the solution can also be written as
\begin{equation}\label{eqsolpresentation}
	u(t)=\int_{0}^{t}\mathcal{R}(t-s)f(u(s))ds+\int_{0}^{t}\int_{D}\mathcal{G}(t-s,x,y)\xi^{H_{1},H_{2}}(dy,ds).
\end{equation}

\begin{theorem}\label{thmsobo}
	Let $u(t)$ be the solution of Eq. \eqref{eqretosol} and $f(u)$ satisfy the assumptions \eqref{eqnonassump}. Assume $2H_{2}+(H_{1}-1)\alpha>0$. Then there holds
	\begin{equation*}
		\mathbb{E}\|A^{\sigma}u(t)\|_{L^{2}(D)}^{2}\leq C,
	\end{equation*}
	where $2\sigma\in[0,\min\{\frac{2H_{2}}{\alpha}+H_{1}-1,2H_{1}+1\})$.
\end{theorem}
\begin{proof}
	According to \eqref{eqsolpresentation}, we have
	\begin{equation*}
		\begin{aligned}
			\mathbb{E}\|A^{\sigma}u(t)\|_{L^{2}(D)}^{2}\leq&C\mathbb{E}\left \|\int_{0}^{t}A^{\sigma}\mathcal{R}(t-s)f(u)ds\right \|_{L^{2}(D)}^{2}\\
			&+C\mathbb{E}\left \|\int_{0}^{t}\int_{D}A^{\sigma}\mathcal{G}(t-s,x,y)\xi^{H_{1},H_{2}}(dy,ds)\right \|_{L^{2}(D)}^{2}\\
			\leq&\uppercase\expandafter{\romannumeral1}+\uppercase\expandafter{\romannumeral2}.
		\end{aligned}
	\end{equation*}
	Consider $\uppercase\expandafter{\romannumeral1}$ first. Using the resolvent estimate $\|(z^{\alpha}+A)^{-1}\|\leq C|z|^{-\alpha}$ for $z\in\Sigma_{\theta}$ \cite{Lubich.1996Ndeefaoaeewapmt}, the Cauchy-Schwarz inequality, and the assumptions \eqref{eqnonassump}, we can obtain
	\begin{equation*}
		\begin{aligned}
			\uppercase\expandafter{\romannumeral1}\leq &C\mathbb{E}\left \|\int_{\Gamma_{\theta,\kappa}}e^{zt}A^{\sigma}z^{\alpha-1}(z^{\alpha}+A)^{-1}\tilde{f}(u)dz\right \|_{L^{2}(D)}^{2}\\
			\leq &C \left(\int_{0}^{t}(t-s)^{-\sigma\alpha}\mathbb{E}\|f(u(s))\|_{L^{2}(D)}ds\right)^{2}\\
			\leq &C \int_{0}^{t}(t-s)^{-2\sigma\alpha+1-\epsilon}\mathbb{E}\|f(u(s))\|_{L^{2}(D)}^{2}ds\\
			\leq&C\left (1+\int_{0}^{t}(t-s)^{-2\sigma\alpha+1-\epsilon}\mathbb{E}\|u\|_{L^{2}(D)}^{2}ds\right ),
		\end{aligned}
	\end{equation*}
	where $\tilde{f}(u)$ means the Laplace transform of $f(u)$ and $-2\sigma\alpha+1>-1$ needs to be satisfied, i.e., $\sigma<\frac{1}{\alpha}$.
	
	According to Theorem \ref{thmisometry} and the definition of $\mathcal{G}$, one has
	\begin{equation*}
		\begin{aligned}
			\uppercase\expandafter{\romannumeral2}\leq& C\sum_{k=1}^{\infty}\int_{0}^{t}\left |{}_{0}\partial^{\frac{1-2H_{2}}{2}}_{t}\lambda_{k}^{\sigma}E_{k}(t)\right |^{2}\|\phi_{k}(y)\|_{H^{\frac{1-2H_{1}}{2}}_{0}(D)}^{2}\|\phi_{k}(x)\|_{L^{2}(D)}^{2}dt\\
			\leq & C\sum_{k=1}^{\infty}\int_{0}^{t}\left |\int_{\Gamma_{\theta,\kappa}}e^{zt}z^{\alpha-1+\frac{1-2H_{2}}{2}}\lambda_{k}^{\sigma}(z^{\alpha}+\lambda_{k})^{-1}dz\right |^{2}\\
			&\qquad\qquad\cdot\|\phi_{k}(y)\|_{H^{\frac{1-2H_{1}}{2}}_{0}(D)}^{2}\|\phi_{k}(x)\|_{L^{2}(D)}^{2}dt.
		\end{aligned}
	\end{equation*}
By the resolvent estimate, Remark \ref{Remeq}, Lemma \ref{thmeigenvalue}, and simple calculations, we have
	\begin{equation*}
		\begin{aligned}
			\uppercase\expandafter{\romannumeral2}\leq &C\sum_{k=1}^{\infty}\int_{0}^{t}\left (\int_{\Gamma_{\theta,\kappa}}|e^{zt}||z|^{\alpha-1+\frac{1-2H_{2}}{2}}|\lambda_{k}^{\sigma+\frac{1-2H_{1}}{4}}(z^{\alpha}+\lambda_{k})^{-1}||dz|\right )^{2}dt\\
			\leq &C\sum_{k=1}^{\infty}\lambda_{k}^{-\frac{1}{2}-2\epsilon}\int_{0}^{t}\left (\int_{\Gamma_{\theta,\kappa}}|e^{zt}||z|^{\alpha-1+\frac{1-2H_{2}}{2}}|\lambda_{k}^{\sigma+\frac{1-H_{1}}{2}+\epsilon}(z^{\alpha}+\lambda_{k})^{-1}||dz|\right )^{2}dt\\
			\leq &C\sum_{k=1}^{\infty}\lambda_{k}^{-\frac{1}{2}-2\epsilon}\int_{0}^{t}\left (\int_{\Gamma_{\theta,\kappa}}|e^{zt}||z|^{(\sigma+\frac{1-H_{1}}{2}+\epsilon)\alpha-1+\frac{1-2H_{2}}{2}}|dz|\right )^{2}dt,
		\end{aligned}
	\end{equation*}
	where we need to require $(\sigma+\frac{1-H_{1}}{2}+\epsilon)\alpha-1+\frac{1-2H_{2}}{2}<-\frac{1}{2}$ and $\sigma+\frac{1-2H_{1}}{2}<1$, i.e., $2\sigma< \min\{\frac{2H_{2}}{\alpha}+H_{1}-1,2H_{1}+1\}$. Combining the Gr\"onwall inequality \cite{Elliott.1992EewsandfafemftCe}, the desired results can be obtained.
\end{proof}

\begin{theorem}\label{thmholder}
	Let $u(t)$ be the solution of Eq. \eqref{eqretosol} and $f(u)$ satisfy the assumptions \eqref{eqnonassump}. Assume $2H_{2}+(H_{1}-1)\alpha>0$. Then there is
	\begin{equation}
		\mathbb{E}\left \|\frac{u(t)-u(t-\tau)}{\tau^{\gamma}}\right \|_{L^{2}(D)}^{2}\leq C,
	\end{equation}
	where $\gamma\in[0,2H_{2}+(H_{1}-1)\alpha)$.
\end{theorem}
\begin{proof}
	According to \eqref{eqsolpresentation}, we have
	\begin{equation*}
		\begin{aligned}
			&\mathbb{E}\left \|\frac{u(t)-u(t-\tau)}{\tau^{\gamma}}\right \|_{L^{2}(D)}^{2}\\
			\leq&C\mathbb{E}\left \|\frac{1}{\tau^{\gamma}}\left (\int_{0}^{t}\mathcal{R}(t-s)f(u)ds-\int_{0}^{t-\tau}\mathcal{R}(t-\tau-s)f(u)ds\right )\right \|_{L^{2}(D)}^{2}\\
			&+C\mathbb{E}\left \|\frac{1}{\tau^{\gamma}}\left (\int_{0}^{t}\int_{D}\mathcal{G}(t-s,x,y)\xi^{H_{1},H_{2}}(dy,ds)\right.\right.\\
			&\qquad\left.\left.-\int_{0}^{t-\tau}\int_{D}\mathcal{G}(t-\tau-s,x,y)\xi^{H_{1},H_{2}}(dy,ds)\right )\right \|_{L^{2}(D)}^{2}\\
			\leq&\uppercase\expandafter{\romannumeral1}+\uppercase\expandafter{\romannumeral2}.
		\end{aligned}
	\end{equation*}
	As for $\uppercase\expandafter{\romannumeral1}$, there holds
	\begin{equation*}
		\begin{aligned}
			\uppercase\expandafter{\romannumeral1}\leq&C\mathbb{E}\left \|\frac{1}{\tau^{\gamma}}\int_{0}^{t-\tau}(\mathcal{R}(t-s)-\mathcal{R}(t-\tau-s))f(u)ds\right \|_{L^{2}(D)}^{2}\\
			&+C\mathbb{E}\left \|\frac{1}{\tau^{\gamma}}\int_{t-\tau}^{t}\mathcal{R}(t-s)f(u)ds\right \|_{L^{2}(D)}^{2}\\
			\leq &\uppercase\expandafter{\romannumeral1}_{1}+\uppercase\expandafter{\romannumeral1}_{2}.
		\end{aligned}
	\end{equation*}
	By the fact $\left |\frac{e^{z\tau}-1}{\tau^{\gamma}}\right |<C|z|^{\gamma}$ with $z\in\Gamma_{\theta,\kappa}$ and $\gamma\in[0,1]$ \cite{Gunzburger.2018ScrotdfstPstaswn} and Theorem \ref{thmsobo}, we have
	\begin{equation*}
		\begin{aligned}
			\uppercase\expandafter{\romannumeral1}_{1}\leq&C\left(\int_{0}^{t-\tau}\left \|\int_{\Gamma_{\theta,\kappa}}e^{z(t-s-\tau)}\frac{e^{z\tau}-1}{\tau^{\gamma}}z^{\alpha-1}(z^{\alpha}+A)^{-1}dz\right \|\mathbb{E}\|f(u)\|_{L^{2}(D)}ds\right)^{2}\\
			\leq&C\left(\int_{0}^{t-\tau}(t-\tau-s)^{-\gamma}(1+\mathbb{E}\|u\|_{L^{2}(D)})dr\right)^{2}\\
			\leq&C,
		\end{aligned}
	\end{equation*}
	where $\gamma\in[0,1)$. Similarly, for $\gamma\in[0,1)$, one can obtain
	\begin{equation*}
		\begin{aligned}
			\uppercase\expandafter{\romannumeral1}_{2}\leq& C\tau^{1-2\gamma}\int_{t-\tau}^{t}\mathbb{E}\|f(u)\|_{L^2(D)}^{2}ds
			\leq C.
		\end{aligned}
	\end{equation*}
	For $\uppercase\expandafter{\romannumeral2}$, we can split it into the following two parts
	\begin{equation*}
		\begin{aligned}
			\uppercase\expandafter{\romannumeral2}\leq& C\mathbb{E}\left\|\frac{1}{\tau^{\gamma}}\left(\int_{0}^{t-\tau}\int_{D}(\mathcal{G}(t-s,x,y)-\mathcal{G}(t-\tau-s,x,y))\xi^{H_{1},H_{2}}(dy,ds)\right)\right \|_{L^{2}(D)}^{2}\\
			& +C\mathbb{E}\left\|\frac{1}{\tau^{\gamma}}\int_{t-\tau}^{t}\int_{D}\mathcal{G}(t-s,x,y)\xi^{H_{1},H_{2}}(dy,ds)\right \|_{L^{2}(D)}^{2}\leq \uppercase\expandafter{\romannumeral2}_{1}+\uppercase\expandafter{\romannumeral2}_{2}.
		\end{aligned}
	\end{equation*}
Using Theorem \ref{thmisometry} and Lemma \ref{thmeigenvalue} yields
	\begin{equation*}
		\begin{aligned}
			\uppercase\expandafter{\romannumeral2}_{1}\leq&C\sum_{k=1}^{\infty}\int_{0}^{t-\tau}\frac{1}{\tau^{2\gamma}}\left|{}_{0}\partial^{\frac{1-2H_{2}}{2}}_{t}(E_{k}(t-s)-E_{k}(t-\tau-s))\right|^{2}\|\phi_{k}(y)\|^{2}_{H^{\frac{1-2H_{1}}{2}}_{0}(D)}ds\\
			\leq& C\sum_{k=1}^{\infty}\int_{0}^{t-\tau}\left|\int_{\Gamma_{\theta,\kappa}}e^{-z(t-\tau-s)}\frac{e^{z\tau}-1}{\tau^{\gamma}}z^{\frac{1-2H_{2}}{2}}\lambda^{\frac{1-2H_{1}}{4}}_{k}z^{\alpha-1}(z^{\alpha}+\lambda_{k})^{-1}dz\right |^{2}ds\\
			\leq& C\sum_{k=1}^{\infty}\lambda_{k}^{-1/2-\epsilon}\int_{0}^{t-\tau}\left(\int_{\Gamma_{\theta,\kappa}}|e^{-z(t-\tau-s)}||z|^{\gamma+\frac{1-2H_{2}}{2}+\frac{1-H_{1}}{2}\alpha-1+\frac{\alpha\epsilon}{2}}|dz|\right )^{2}ds\\
			\leq& C\sum_{k=1}^{\infty}\lambda_{k}^{-1/2-\epsilon}\int_{0}^{t-\tau}s^{(H_{1}-1)\alpha+2H_{2}-1-2\gamma-\alpha\epsilon}ds,
		\end{aligned}
	\end{equation*}
	where we need to require $2\gamma<2H_{2}+(H_{1}-1)\alpha$. Similarly, for $\uppercase\expandafter{\romannumeral2}_{2}$, one has
	\begin{equation*}
		\begin{aligned}
			\uppercase\expandafter{\romannumeral2}_{2}\leq& C\frac{1}{\tau^{2\gamma}}\sum_{k=1}^{\infty}\int_{t-\tau}^{t}\left({}_{0}\partial^{\frac{1-2H_{2}}{2}}_{t}\lambda^{\frac{1-2H_{1}}{4}}_{k}E_{k}(t-s)\right)^{2}ds\\
			\leq& C\frac{1}{\tau^{2\gamma}}\sum_{k=1}^{\infty}\lambda_{k}^{-1/2-\epsilon}\int_{t-\tau}^{t}\left(\int_{\Gamma_{\theta,\kappa}}|e^{-z(t-s)}||z|^{\frac{1-2H_{2}}{2}+\frac{1-H_{1}}{2}\alpha-1+\frac{\alpha\epsilon}{2}}|dz|\right )^{2}ds\\
			\leq& C\frac{1}{\tau^{2\gamma}}\sum_{k=1}^{\infty}\lambda_{k}^{-1/2-\epsilon}\int_{0}^{\tau}s^{(H_{1}-1)\alpha+2H_{2}-1-\alpha\epsilon}ds,
		\end{aligned}
	\end{equation*}
	where we need to require $2\gamma<2H_{2}+(H_{1}-1)\alpha$. Collecting the above estimates leads to the desired result.
\end{proof}
\section{Wong-Zakai Approximation}\label{sec3}
In this section, we use the Wong-Zakai approximation \cite{EugeneWong.1965OtCoOItSI,Eugene.1965Otrboasde} to regularize the fractional Brownian sheet noise and provide a systematic approach to prove the convergence of the Wong-Zakai approximation.

Here, we introduce the Wong-Zakai approximation first. Let $\tau=T/M$ and $h=l/N$. Denote $I_{i}=(t_{i},t_{i+1}]$ and $D_{j}=(x_{j},x_{j+1}]$ with $t_{i}=i\tau$ ($i=0,1,\ldots,M$) and $x_{j}=jh$ ($j=0,1,\ldots,N$). The Wong-Zakai approximation of $\xi^{H_{1},H_{2}}(x,t)$ can be written as
\begin{equation}\label{eqdefxiR}
	\xi_{R}^{H_{1},H_{2}}(x,t)=\sum_{i=0}^{M-1}\sum_{j=0}^{N-1}\left(\frac{1}{\tau  h}\int_{I_{i}}\int_{D_{j}}\xi^{H_{1},H_{2}}(dy,ds)\right)\chi_{I_{i}\times D_{j}}(t,x),
\end{equation}
where $\chi_{I_{i}\times D_{j}}(t,x)$ is the characteristic function on $I_{i}\times D_{j}$.
Then we introduce $u_{R}(x,t)$ as the solution of the following regularized equation, i.e.,
\begin{equation}\label{eqregularizeeq}
	\left\{\begin{aligned}
		&\partial_{t}u_{R}(x,t)+{}_{0}\partial_{t}^{1-\alpha}A u_{R}(x,t)=f(u_{R})+\xi_{R}^{H_{1},H_{2}}(x,t)\qquad (x,t)\in{D}\times (0,T],\\
		&u_{R}(x,0)=0\qquad x\in{D},\\
		&u_{R}(x,t)=0\qquad (x,t)\in\partial{D}\times (0,T].
	\end{aligned}\right.
\end{equation}
Simple calculations lead to
\begin{equation}\label{eqregsolpre}
	u_{R}(x,t)=\int_{0}^{t}\mathcal{R}(t-s)f(u_{R})ds+\int_{0}^{t}\int_{D}\mathcal{G}_{R}(t-s,x,y)\xi^{H_{1},H_{2}}(dy,ds).
\end{equation}
Here, $\mathcal{G}_{R}(t,x,y)$ is defined by
\begin{equation*}
	\begin{aligned}
		\mathcal{G}_{R}(t,x,y)=&\sum_{i=0}^{M-1}\sum_{j=0}^{N-1}\frac{\chi_{I_{i}\times D_{j}}(t,y)}{\tau h}\int_{I_{i}}\int_{D_{j}}\mathcal{G}(t,x,y)dtdy\\
		=&\sum_{k=1}^{\infty}E_{R,k}(t)\phi_{k}(x)\phi_{R,k}(y),
	\end{aligned}
\end{equation*}
where
\begin{equation*}
	E_{R,k}(t)=\frac{1}{\tau}\sum_{i=0}^{M-1}\chi_{I_{i}}(t)\int_{I_{i}}E_{k}(t)dt,\qquad
	\phi_{R,k}(y)=\frac{1}{h}\sum_{j=0}^{N-1}\chi_{I_{j}}(y)\int_{D_{j}}\phi_{k}(y)dt.
\end{equation*}

\begin{lemma}\label{lemEkest}
	Let $E_{k}(t)$ be defined in \eqref{eqdefEk}. For $k>0$, $\sigma\in[0,1]$, and $\gamma\geq 0$, if $\gamma+\sigma\alpha<\frac{1}{2}$, then there exists a uniform constant $C$ such that
	\begin{equation}
		\|\lambda^{\sigma}_{k}E_{k}(t)\|_{H^{\gamma}((0,T))}=\|{}_{0}\partial^{\gamma}_{t}\lambda^{\sigma}_{k}E_{k}(t)\|_{L^{2}((0,T))}\leq C.
	\end{equation}
\end{lemma}
\begin{proof}
	By the definition of $E_{k}(t)$ and the resolvent estimate \cite{Lubich.1996Ndeefaoaeewapmt}, we have
	\begin{equation*}
		\begin{aligned}
			\|{}_{0}\partial^{\gamma}_{t}\lambda^{\sigma}_{k}E_{k}(t)\|_{L^{2}((0,T))}^{2}\leq&C\int_{0}^{T}\left(\int_{\Gamma_{\theta,\kappa}}|e^{zt}|\lambda^{\sigma}_{k}|z|^{\gamma+\alpha-1}|(z^{\alpha}+\lambda_{k})|^{-1}|dz|\right)^{2}dt\\
			\leq&C\int_{0}^{T}\left(\int_{\Gamma_{\theta,\kappa}}|e^{zt}||z|^{\gamma+\sigma\alpha-1}|dz|\right)^{2}dt\\
			\leq &C\int_{0}^{T}t^{-2\gamma-2\sigma\alpha}dt,
		\end{aligned}
	\end{equation*}
	where $C$ is a positive constant independent of $k$. To preserve the boundedness of $\|{}_{0}\partial^{\gamma}_{t}\lambda^{\sigma}_{k}E_{k}(t)\|_{L^{2}((0,T))}^{2}$, we need to require $\gamma<\frac{1}{2}-\sigma\alpha$.
\end{proof}

\begin{theorem}\label{thmurB}
	Let $u_{R}(t)$ be the solution of \eqref{eqregularizeeq} and $f(u)$ satisfy the assumptions \eqref{eqnonassump}. Assume $2H_{2}+(H_{1}-1)\alpha>0$. Then we have
	\begin{equation*}
		\mathbb{E}\|A^{\sigma}u_{R}(t)\|_{L^{2}(D)}^{2}\leq C,
	\end{equation*}
	where $2\sigma\in[0,\min\{\frac{2H_{2}}{\alpha}+H_{1}-1,2H_{1}+1\})$.
\end{theorem}
\begin{proof}
	Here we split $\mathbb{E}\|A^{\sigma}u_{R}(t)\|_{L^{2}(D)}^{2}$ into two parts, i.e.,
	\begin{equation*}
		\begin{aligned}
			\mathbb{E}\|A^{\sigma}u_{R}(t)\|_{L^{2}(D)}^{2}\leq &C\mathbb{E}\left \|\int_{0}^{t}A^{\sigma}\mathcal{R}(t-s)f(u_{R})ds\right \|_{L^{2}(D)}^{2}\\
			&+C\mathbb{E}\left \|\int_{0}^{t}\int_{D}A^{\sigma}\mathcal{G}_{R}(t-s,x,y)\xi^{H_{1},H_{2}}(dy,ds)\right \|_{L^{2}(D)}^{2}\\
			\leq & \uppercase\expandafter{\romannumeral1}+\uppercase\expandafter{\romannumeral2}.
		\end{aligned}
	\end{equation*}
	Similar to the proof of Theorem \ref{thmsobo}, one has
	\begin{equation}
		\uppercase\expandafter{\romannumeral1}\leq C\left (1+\int_{0}^{t}(t-s)^{-2\sigma\alpha+1-\epsilon}\mathbb{E}\|u_{R}\|_{L^{2}(D)}^{2}ds\right ).
	\end{equation}
	Using the standard approximation theory \cite{Brenner.2008TMToFEM} and $H_{1},~H_{2}\in(0,\frac{1}{2}]$, we have
	\begin{equation}\label{eqErkest}
		\begin{aligned}
			\left \|{}_{0}\partial^{\frac{1-2H_{2}}{2}}_{t}E_{R,k}(t)\right \|_{L^{2}((0,T))}^{2}
			\leq& C\left \|{}_{0}\partial^{\frac{1-2H_{2}}{2}}_{t}(E_{R,k}(t)-E_{k}(t))\right \|_{L^{2}((0,T))}^{2}\\
			&+C\left \|{}_{0}\partial^{\frac{1-2H_{2}}{2}}_{t}E_{k}(t)\right \|_{L^{2}((0,T))}^{2}\\
			\leq & C\left \|{}_{0}\partial^{\frac{1-2H_{2}}{2}}_{t}E_{k}(t)\right \|_{L^{2}((0,T))}^{2}
		\end{aligned}
	\end{equation}
	and
	\begin{equation}\label{eqphiRkest}
		\begin{aligned}
			\|\phi_{R,k}(y)\|_{H^{\frac{1-2H_{1}}{2}}_{0}({D})}^{2}
			\leq& C\|\phi_{R,k}(y)-\phi_{k}(y)\|_{H^{\frac{1-2H_{1}}{2}}_{0}({D})}^{2}+C\|\phi_{k}(y)\|_{H^{\frac{1-2H_{1}}{2}}_{0}({D})}^{2}\\
			\leq&C\|\phi_{k}(y)\|_{H^{\frac{1-2H_{1}}{2}}_{0}({D})}^{2}.
		\end{aligned}
	\end{equation}
	Thus, combining Theorem \ref{thmisometry}, one has
	\begin{equation*}
		\begin{aligned}
			\uppercase\expandafter{\romannumeral2}\leq&C\sum_{k=1}^{\infty}\lambda^{2\sigma}_{k}\left \|{}_{0}\partial^{\frac{1-2H_{2}}{2}}_{t}E_{k}(t)\right \|_{L^{2}((0,T))}^{2}\|\phi_{k}(y)\|_{H^{\frac{1-2H_{1}}{2}}_{0}({D})}^{2}.
		\end{aligned}
	\end{equation*}
	Similar to the proof of Theorem \ref{thmsobo}, we can get
	\begin{equation*}
		\uppercase\expandafter{\romannumeral2}\leq C
	\end{equation*}
	with $2\sigma\in[0,\min\{\frac{2H_{2}}{\alpha}+H_{1}-1,2H_{1}+1\})$. Combining above estimates and using  Gr\"{o}nwall inequality \cite{Elliott.1992EewsandfafemftCe,Nie.2020NaftsfFswtis} lead to the desired results.
\end{proof}

\begin{theorem}\label{thmreglarerr}
	Let $u(t)$ and $u_{R}(t)$ be the solutions of Eqs. \eqref{eqretosol} and \eqref{eqregularizeeq}, respectively. Assume $2H_{2}+(H_{1}-1)\alpha>0$. Then there exists a constant $C$ such that
	\begin{equation*}
		\mathbb{E}\|u(t)-u_{R}(t)\|_{L^{2}(D)}^{2}\leq C(h^{2\sigma+2H_{1}-1}+\tau^{2H_{2}-\frac{\alpha}{2}-\epsilon}h^{2H_{1}-1}),
	\end{equation*}
	where $\sigma\in(\frac{1-2H_{1}}{2}
	,\min\{\frac{2H_{2}}{\alpha}-\frac{1}{2},1+\epsilon\})$.
\end{theorem}

\begin{proof}
	According to \eqref{eqsolpresentation} and \eqref{eqregsolpre}, one can get
	\begin{equation*}
		\begin{aligned}
			&\mathbb{E}\|u(t)-u_{R}(t)\|_{L^{2}(D)}^{2}\\
			\leq&C\mathbb{E}\left\|\int_{0}^{t}\mathcal{R}(t-s)(f(u)-f(u_{R}))ds\right\|_{L^{2}(D)}^{2}\\
			&+C\mathbb{E}\left \|\int_{0}^{t}\int_{D}(\mathcal{G}(t-s,x,y)-\mathcal{G}_{R}(t-s,x,y))\xi^{H_{1},H_{2}}(dy,ds)\right \|_{L^{2}(D)}^{2} \\
			\leq&\uppercase\expandafter{\romannumeral1}+\uppercase\expandafter{\romannumeral2}.
		\end{aligned}
	\end{equation*}
	Assumptions \eqref{eqnonassump} and the resolvent estimate lead to
	\begin{equation*}
		\begin{aligned}
			\uppercase\expandafter{\romannumeral1}\leq &C\mathbb{E}\left \|\int_{\Gamma_{\theta,\kappa}}e^{zt}z^{\alpha-1}(z^{\alpha}+A)^{-1}(\tilde{f}(u)-\tilde{f}(u_{R}))dz\right \|_{L^{2}(D)}^{2}\\
			\leq &C\mathbb{E} \left(\int_{0}^{t}\int_{\Gamma_{\theta,\kappa}}|e^{z(t-s)}|\|z^{\alpha-1}(z^{\alpha}+A)^{-1}\||dz|\|f(u(s))-f(u_{R}(s))\|_{L^{2}(D)}ds\right)^{2}\\
			\leq &C \int_{0}^{t}\mathbb{E}\|f(u)-f(u_{R})\|_{L^{2}(D)}^{2}ds\\
			\leq&C\int_{0}^{t}\mathbb{E}\|u-u_{R}\|^{2}_{L^{2}(D)}ds.
		\end{aligned}
	\end{equation*}
	As for $\uppercase\expandafter{\romannumeral2}$, using the following fact
	\begin{equation*}
		\begin{aligned}
			E_{k}(t)\phi_{k}(y)-E_{R,k}(t)\phi_{R,k}(y)=&(E_{k}(t)\phi_{k}(y)-E_{k}(t)\phi_{R,k}(y))\\
			&+(E_{k}(t)\phi_{R,k}(y)-E_{R,k}(t)\phi_{R,k}(y))
		\end{aligned}
	\end{equation*}
	and Theorem \ref{thmisometry}, one can get
	\begin{equation*}
		\begin{aligned}
			\uppercase\expandafter{\romannumeral2}
			\leq&C\sum_{k=1}^{\infty}\|{}_{0}\partial^{\frac{1-2H_{2}}{2}}_{t}E_{k}(t)\|_{L^{2}((0,T))}^{2}\|\phi_{k}(y)-\phi_{R,k}(y)\|^{2}_{H^{\frac{1-2H_{1}}{2}}_{0}({D})}\\
			+&C\sum_{k=1}^{\infty}\|{}_{0}\partial^{\frac{1-2H_{2}}{2}}_{t}(E_{k}(t)-E_{R,k}(t))\|_{L^{2}((0,T))}^{2}\|\phi_{R,k}(y)\|^{2}_{H^{\frac{1-2H_{1}}{2}}_{0}({D})}\\
			\leq&\uppercase\expandafter{\romannumeral2}_{1}+\uppercase\expandafter{\romannumeral2}_{2}.
		\end{aligned}
	\end{equation*}
	Combining Lemma \ref{lemEkest} and standard approximation theory \cite{Brenner.2008TMToFEM}, one has
	\begin{equation*}
		\begin{aligned}
			\uppercase\expandafter{\romannumeral2}_{1}\leq &C\sum_{k=1}^{\infty}\|{}_{0}\partial^{\frac{1-2H_{2}}{2}}_{t}E_{k}(t)\|^{2}_{L^{2}((0,T) )}\|\phi_{k}(y)-\phi_{R,k}(y)\|_{H^{\frac{1-2H_{1}}{2}}_{0}({D})}^{2}\\
			\leq &C\sum_{k=1}^{\infty}\lambda^{-\frac{1}{2}-\epsilon}_{k}\|{}_{0}\partial^{\frac{1-2H_{2}}{2}}_{t}\lambda^{\frac{1+2\epsilon}{4}}_{k}E_{k}(t)\|^{2}_{L^{2}((0,T) )}\|\phi_{k}(y)-\phi_{R,k}(y)\|_{H^{\frac{1-2H_{1}}{2}}_{0}({D})}^{2}\\
			\leq &Ch^{2\sigma+2H_{1}-1}\sum_{k=1}^{\infty}\lambda^{-\frac{1}{2}-\epsilon}_{k}\|{}_{0}\partial^{\frac{1-2H_{2}}{2}}_{t}\lambda^{\frac{1+2\epsilon+2\sigma}{4}}_{k}E_{k}(t)\|^{2}_{L^{2}((0,T) )},
		\end{aligned}
	\end{equation*}
	where we need to require $\frac{1}{2}-\frac{1+2\epsilon+2\sigma}{4}\alpha>\frac{1-2H_{2}}{2}$ and $2\sigma+2H_{1}-1>0$ to preserve the boundedness of $\uppercase\expandafter{\romannumeral2}_{1}$, i.e., $\sigma\in(\frac{1-2H_{1}}{2}
	,\min\{\frac{2H_{2}}{\alpha}-\frac{1}{2},1+\epsilon\})$. Similarly, by the inverse estimate and projection theorem \cite{Brenner.2008TMToFEM}, we can also obtain
	\begin{equation*}
		\begin{aligned}
			\uppercase\expandafter{\romannumeral2}_{2}\leq &C\sum_{k=1}^{\infty}\|{}_{0}\partial^{\frac{1-2H_{2}}{2}}_{t}(E_{k}(t)-E_{R,k}(t))\|^{2}_{L^{2}((0,T) )}\|\phi_{R,k}(y)\|_{H^{\frac{1-2H_{1}}{2}}_{0}({D})}^{2}\\
			\leq &Ch^{2H_{1}-1}\sum_{k=1}^{\infty}\lambda_{k}^{-1/2-2\epsilon}\|{}_{0}\partial^{\frac{1-2H_{2}}{2}}_{t}\lambda_{k}^{1/4+\epsilon}(E_{k}(t)-E_{R,k}(t))\|^{2}_{L^{2}((0,T) )}\\
			\leq &C\tau^{2H_{2}-\frac{\alpha}{2}-\epsilon_{0}}h^{2H_{1}-1}.
		\end{aligned}
	\end{equation*}
	Thus the desired results can be achieved by the Gr\"onwall inequality directly.
\end{proof}

\begin{remark}
	When taking $H_{2}=\frac{1}{2}$ and $\alpha=1$ in Theorem \ref{thmreglarerr}, there exists 
	\begin{equation*}
		\mathbb{E}\|u(t)-u_{R}(t)\|_{L^{2}(D)}^{2}\leq C(h^{2H_{1}-\epsilon}+\tau^{\frac{1}{4}-\epsilon}h^{2H_{1}-1}),
	\end{equation*}
	which can recover the results provided in \cite{Cao.2017Aseewawarn}.
\end{remark}

\section{Numerical scheme and error analysis}\label{sec4}
In this section, we construct the numerical scheme for the above regularized equation \eqref{eqregularizeeq}. Here we use the finite element and backward Euler convolution quadrature methods  to discretize the spatial and temporal operators, respectively. At the same time, we provide the corresponding error analysis.
\subsection{Numerical scheme}
Let $\mathcal{T}_{h}$ be a shape regular quasi-uniform partition of the domain $D$. Introduce $X_{h}\subset H^{1}_{0}(D)$ as the space of continuous piecewise linear function over the $\mathcal{T}_{h}$.  Define the $L^{2}$ orthogonal projection $P_{h}:~L^{2}(D)\rightarrow X_{h}$:
\begin{equation*}
	\begin{aligned}
		&(P_{h}u,v_{h})=(u,v_{h})\quad \forall u\in L^{2}(D)~~\forall v_{h}\in X_{h},\\
	\end{aligned}
\end{equation*}
Denote $A_{h}$ as $(A_{h}u_{h},v_{h})=	(\nabla u_{h},\nabla v_{h})$ with $ u_{h}, v_{h}\in X_{h}$.
The semi-discrete Galerkin scheme for \eqref{eqregularizeeq} can be written as: find $u_{h}\in X_{h}$ such that
\begin{equation*}
	(\partial_{t}u_{h},v_{h})+({}_{0}\partial ^{1-\alpha}_{t}A_{h}u_{h},v_{h})=(f(u_{h}),v_{h})+(\xi^{H_{1},H_{2}}_{R},v_{h})
\end{equation*}
with $u_{h}(0)=0$. It can also be written as
\begin{equation}\label{eqsemischm1}
	\partial_{t}u_{h}+{}_{0}\partial ^{1-\alpha}_{t}A_{h}u_{h}=P_{h}f(u_{h})+P_{h}\xi^{H_{1},H_{2}}_{R}.
\end{equation}
Then we use the backward Euler convolution quadrature method to discretize the temporal derivative operator, i.e., the fully discrete scheme is
\begin{equation}\label{eqfullscheme}
	\frac{u^{n}_{h}-u^{n-1}_{h}}{\tau}+\sum_{i=0}^{n-1}d^{(1-\alpha)}_{i}A_{h}u^{n-i}_{h}=P_{h}f(u^{n-1}_{h})+P_{h}\xi^{H_{1},H_{2}}_{R,n},
\end{equation}
where $\xi^{H_{1},H_{2}}_{R,n}=\xi^{H_{1},H_{2}}_{R}(t_{n})$ and
\begin{equation}\label{eqfullschm}
	\sum_{i=0}^{\infty}d^{(\alpha)}_{i}\zeta^{i}=(\delta_{\tau}(\zeta))^{\alpha},\quad \delta_{\tau}(\zeta)=\frac{1-\zeta}{\tau}.
\end{equation}
Next, we provide the temporal and spatial error analyses, respectively.

\subsection{Spatial error analysis}
Now, we suppose that $\bar{u}_{h}\in X_{h}$ is the solution of the equation
\begin{equation}\label{eqbaru0}
	(\partial_{t}\bar{u}_{h},v_{h})+({}_{0}\partial ^{1-\alpha}_{t}A_{h}\bar{u}_{h},v_{h})=(f(\bar{u}_{h}),v_{h})+(\xi^{H_{1},H_{2}}_{RS},v_{h})
\end{equation}
with $u_{h}(0)=0$ and
\begin{equation*}
	\xi^{H_{1},H_{2}}_{RS}(t)=\sum_{j=0}^{N-1}\left(\frac{1}{h}\int_{D_{j}}\xi^{H_{1},H_{2}}(dy,t)\right)\chi_{D_{j}}(x).
\end{equation*}
Then Eq. \eqref{eqbaru0} can  also be written as
\begin{equation*}
	\partial_{t}\bar{u}_{h}+{}_{0}\partial ^{1-\alpha}_{t}A_{h}\bar{u}_{h}=P_{h}f(\bar{u}_{h})+P_{h}\xi^{H_{1},H_{2}}_{RS}.
\end{equation*}
Introduce the Laplace transform of the operator $\mathcal{R}_{h}$ as
\begin{equation*}
	\tilde{\mathcal{R}}_{h}(z)=z^{\alpha-1}(z^{\alpha}+A_{h})^{-1},
\end{equation*}
which leads to
\begin{equation}\label{eqsemischsol}
	\bar{u}_{h}(t)=\int_{0}^{t}\mathcal{R}_{h}(t-s)P_{h}f(\bar{u}_{h}(s))ds+\int_{0}^{t}\mathcal{R}_{h}(t-s)P_{h}\xi_{RS}^{H_{1},H_{2}}(ds).
\end{equation}
Let
\begin{equation}\label{eqdefGRh}
	\mathcal{G}_{R,h}(t,x,y)=\sum_{k=1}^{\infty}\mathcal{G}_{R,h,k}(t,x,y)
\end{equation}
with
\begin{equation}\label{eqdefGRhk}
	\mathcal{G}_{R,h,k}(t,x,y)=\mathcal{R}_{h}(t)P_{h}\phi_{k}(x)\phi_{R,k}(y).
\end{equation}
Thus by \eqref{eqsemischsol}, the solution $\bar{u}_{h}$ can also be written as
\begin{equation}\label{eqsemisol}
	\bar{u}_{h}(t)=\int_{0}^{t}\mathcal{R}_{h}(t-s)P_{h}f(\bar{u}_{h}(s))ds+\int_{0}^{t}\int_{D}\mathcal{G}_{R,h}(t-s,x,y)\xi^{H_{1},H_{2}}(dy,ds).
\end{equation}
Then we provide the properties of $\bar{u}_{h}(t)$.

\begin{theorem}\label{thmsobobaru}
	Let $\bar{u}_{h}(t)$ be the solution of Eq. \eqref{eqbaru0} and $f(u)$ satisfy \eqref{eqnonassump}. Assume $2H_{2}+(H_{1}-1)\alpha>0$. Then we have
	\begin{equation*}
		\mathbb{E}\|\bar{u}_{h}(t)\|_{L^2(D)}^{2}\leq C
	\end{equation*}
	and
	\begin{equation}
		\mathbb{E}\left \|\frac{\bar{u}_{h}(t)-\bar{u}_{h}(t-\tau)}{\tau^{\gamma}}\right \|_{L^{2}(D)}^{2}\leq C,
	\end{equation}
	where $\gamma\in(0,2H_{2}+(H_{1}-1)\alpha)$.
\end{theorem}
\begin{proof}
	According to \eqref{eqsemisol}, we have
	\begin{equation*}
		\begin{aligned}
			\mathbb{E}\|\bar{u}_{h}(t)\|_{L^2(D)}^{2}\leq&C\mathbb{E}\left\|\int_{0}^{t}\mathcal{R}_{h}(t-s)P_{h}f(\bar{u}_{h}(s))ds\right \|_{L^2(D)}^{2}\\
			&+C\mathbb{E}\left\|\int_{0}^{t}\int_{D}\mathcal{G}_{R,h}(t-s,x,y)\xi^{H_{1},H_{2}}(dy,ds)\right \|_{L^2(D)}^{2}\\
			\leq&\uppercase\expandafter{\romannumeral1}+\uppercase\expandafter{\romannumeral2}.
		\end{aligned}
	\end{equation*}
	Similar to the derivations of Theorem \ref{thmsobo}, one has
	\begin{equation*}
		\uppercase\expandafter{\romannumeral1}\leq C\left (1+\int_{0}^{t}\mathbb{E}\|\bar{u}_{h}(t)\|_{L^{2}(D)}^{2}ds\right ).
	\end{equation*}\\
	As for $\uppercase\expandafter{\romannumeral2}$, Theorem \ref{thmsobo}, Lemma \ref{lemEkest}, and Eq. \eqref{eqphiRkest} give
	\begin{equation*}
		\begin{aligned}
			\uppercase\expandafter{\romannumeral2}\leq&C\sum_{k=1}^{\infty}\left \|{}_{0}\partial^{\frac{1-2H_{2}}{2}}_{t}E_{k}(t)\right \|_{L^{2}((0,T))}^{2}\|\phi_{k}(y)\|_{H^{\frac{1-2H_{1}}{2}}_{0}({D})}^{2}\\
			\leq& C.
		\end{aligned}
	\end{equation*}
	Thus we can get the first desired result by using the Gr\"{o}nwall inequality.  As for the second estimate, using the stability of $P_{h}$, i.e., $\|P_{h}u\|_{L^{2}(D)}\leq \|u\|_{L^{2}(D)}$ \cite{Thomee.2006GFEMfPP} and the similar arguments in Theorem \ref{thmholder}, we can obtain the desired result.

\end{proof}

\begin{theorem}\label{thm:spa}
	Let $u(t)$ and $\bar{u}_{h}(t)$ be the solutions of Eqs. \eqref{eqretosol} and \eqref{eqbaru0}, respectively. Assume that $f(u)$ satisfies \eqref{eqnonassump} and $2H_{2}+(H_{1}-1)\alpha>0$. Then we have
	\begin{equation*}
		\mathbb{E}\|u(t)-\bar{u}_{h}(t)\|^{2}_{L^{2}(D)}\leq Ch^{2\sigma+2H_{1}-1},
	\end{equation*}
	where $\sigma\in(\frac{1-2H_{1}}{2}
	,\min\{\frac{2H_{2}}{\alpha}-\frac{1}{2},1+\epsilon\})$.
\end{theorem}
\begin{proof}
	According to \eqref{eqsolpresentation} and \eqref{eqsemischsol}, one obtains
	\begin{equation*}
		\begin{aligned}
			&\mathbb{E}\|u(t)-\bar{u}_{h}(t)\|_{L^{2}(D)}^{2}\\
			\leq&C\mathbb{E}\left \|\int_{0}^{t}\mathcal{R}(t-s)f(u(s))ds-\int_{0}^{t}\mathcal{R}_{h}(t-s)P_{h}f(\bar{u}_{h}(s))ds\right \|_{L^{2}(D)}^{2}\\
			&+C\mathbb{E}\left \|\int_{0}^{t}\int_{D}(\mathcal{G}(t-s,x,y)-\mathcal{G}_{R,h}(t-s,x,y))\xi^{H_{1},H_{2}}(dy,ds)\right \|_{L^{2}(D)}^{2}\\
			\leq& \uppercase\expandafter{\romannumeral1}+\uppercase\expandafter{\romannumeral2}.
		\end{aligned}
	\end{equation*}
	As for $\uppercase\expandafter{\romannumeral1}$, one has
	\begin{equation*}
		\begin{aligned}
			\uppercase\expandafter{\romannumeral1}\leq& \mathbb{E}\left \|\int_{0}^{t}\mathcal{R}(t-s)f(u(s))ds-\int_{0}^{t}\mathcal{R}_{h}(t-s)P_{h}f(u(s))ds\right \|_{L^{2}(D)}^{2}\\
			&+\mathbb{E}\left \|\int_{0}^{t}\mathcal{R}_{h}(t-s)P_{h}f(u(s))ds-\int_{0}^{t}\mathcal{R}_{h}(t-s)P_{h}f(\bar{u}_{h}(s))ds\right \|_{L^{2}(D)}^{2}\\
			\leq&\uppercase\expandafter{\romannumeral1}_{1}+\uppercase\expandafter{\romannumeral1}_{2}.
		\end{aligned}
	\end{equation*}
	Using the assumptions \eqref{eqnonassump}, the fact $\|(\mathcal{R}(t-s)-\mathcal{R}_{h}(t-s)P_{h})\|\leq Ch^{2}|z|^{\alpha-1}$, and Theorem \ref{thmurB}, we have
	\begin{equation*}
		\begin{aligned}
			\uppercase\expandafter{\romannumeral1}_{1}
			\leq&\left( \int_{0}^{t}\|(\mathcal{R}(t-s)-\mathcal{R}_{h}(t-s)P_{h})\|\mathbb{E}\|f(u(s))\|_{L^{2}(D)}ds\right) ^{2}\\
			\leq&Ch^{4}\left( \int_{0}^{t}(t-s)^{-\alpha}(1+\mathbb{E}\|u(s)\|_{L^{2}(D)})ds\right) ^{2}\\
			\leq& Ch^{4}.
		\end{aligned}
	\end{equation*}
	As for $\uppercase\expandafter{\romannumeral1}_{2}$, the stability of $P_{h}$, the assumptions \eqref{eqnonassump}, and the Cauchy-Schwarz inequality imply
	\begin{equation*}
		\begin{aligned}
			\uppercase\expandafter{\romannumeral1}_{2}\leq &C\mathbb{E}\left (\int_{0}^{t}\|\mathcal{R}_{h}(t-s)\|\|(f(u(s))-f(\bar{u}_{h}(s)))\|_{L^{2}(D)}ds\right) ^{2}\\
			\leq &C\int_{0}^{t}\mathbb{E}\|u(s)-\bar{u}_{h}(s)\|_{L^{2}(D)}^{2}ds.
		\end{aligned}
	\end{equation*}
	As for $\uppercase\expandafter{\romannumeral2}$, one has
	\begin{equation*}
		\begin{aligned}
			\uppercase\expandafter{\romannumeral2}\leq &C\mathbb{E}\left \|\int_{0}^{t}\int_{D}(\mathcal{G}(t-s,x,y)-\mathcal{G}_{Rs}(t-s,x,y))\xi^{H_{1},H_{2}}(dy,ds)\right \|_{L^{2}(D)}^{2}\\
			&+C\mathbb{E}\left \|\int_{0}^{t}\int_{D}(\mathcal{G}_{Rs}(t-s,x,y)-\mathcal{G}_{R,h}(t-s,x,y))\xi^{H_{1},H_{2}}(dy,ds)\right \|_{L^{2}(D)}^{2}\\
			\leq&\uppercase\expandafter{\romannumeral2}_{1}+\uppercase\expandafter{\romannumeral2}_{2},
		\end{aligned}
	\end{equation*}
	where
	\begin{equation*}
		\mathcal{G}_{Rs}(t,x,y)=\sum_{k=1}^{\infty}\mathcal{R}(t)\phi_{k}(x)\phi_{R,k}(y).
	\end{equation*}
	Similar to the derivations of $\uppercase\expandafter{\romannumeral2}_{1}$ in the proof of Theorem \ref{thmreglarerr}, we have
	\begin{equation*}
		\uppercase\expandafter{\romannumeral2}_{1}\leq Ch^{2\sigma+2H_{1}-1},
	\end{equation*}
	where $\sigma\in(\frac{1-2H_{1}}{2}
	,\min\{\frac{2H_{2}}{\alpha}-\frac{1}{2},1+\epsilon\})$.
	As for $\uppercase\expandafter{\romannumeral2}_{2}$, the inverse estimate gives
	\begin{equation*}
		\begin{aligned}
			\uppercase\expandafter{\romannumeral2}_{2}\leq&C\sum_{k=1}^{\infty}\left \|{}_{0}\partial^{\frac{1-2H_{2}}{2}}_{t}(\mathcal{R}(t)-\mathcal{R}_{h}(t)P_{h})\phi_{k}(x)\right \|_{L^{2}(0,t,L^{2}(D))}^{2}\|\phi_{R,k}(y)\|_{H^{\frac{1-2H_{1}}{2}}(D)}^{2}\\
			\leq&Ch^{2H_{1}-1}\sum_{k=1}^{\infty}\lambda_{k}^{-\frac{1}{2}-2\epsilon}\left \|{}_{0}\partial^{\frac{1-2H_{2}}{2}}_{t}(\mathcal{R}(t)-\mathcal{R}_{h}(t)P_{h})A^{\frac{1}{4}+\epsilon}\phi_{k}(x)\right\|_{L^{2}(0,t,L^{2}(D))}^{2}.
		\end{aligned}
	\end{equation*}
	Using the fact $\|(\mathcal{R}(t)-\mathcal{R}_{h}(t)P_{h})A^{s}\|\leq Ch^{2-2s}$ for $s\in[0,\frac{1}{2}]$ \cite{Thomee.2006GFEMfPP} and the interpolation property \cite{Adams.2003SS}, one can obtain
	\begin{equation*}
		\begin{aligned}
			&\left \|{}_{0}\partial^{\frac{1-2H_{2}}{2}}_{t}(\mathcal{R}(t)-\mathcal{R}_{h}(t)P_{h})A^{\frac{1}{4}+\epsilon}\phi_{k}(x)\right\|_{L^{2}(0,t,L^{2}(D))}^{2}\\
			\leq&C\int_{0}^{t}\left \|\int_{\Gamma_{\theta,\kappa}}e^{zt}z^{\frac{1-2H_{2}}{2}}z^{\alpha-1}\left ((z^{\alpha}+A)^{-1}-(z^{\alpha}+A_{h})^{-1}P_{h}\right )A^{\frac{1}{4}+\epsilon}\phi_{k}(x)dz\right \|_{L^{2}(D)}^{2}dt\\
			\leq&Ch^{\frac{4H_{2}}{\alpha}-1-\epsilon_{0}}\int_{0}^{t}\left(\int_{\Gamma_{\theta,\kappa}}|e^{zt}||z|^{-\frac{1}{2}-\epsilon}|dz|\right)^{2}dt\\
			\leq&Ch^{\frac{4H_{2}}{\alpha}-1-\epsilon_{0}},
		\end{aligned}
	\end{equation*}
	which leads to
	\begin{equation*}
		\uppercase\expandafter{\romannumeral2}_{2}\leq Ch^{\frac{4H_{2}}{\alpha}+2H_{1}-2-\epsilon}.
	\end{equation*}
	Therefore, we complete the proof.
\end{proof}
\subsection{Temporal error analysis}
In this subsection, we consider the temporal  error estimate, that is, the estimate of $\mathbb{E}\|u^{n}_{h}-\bar{u}_{h}(t_{n})\|_{L^{2}(D)}$.

In the following,  we provide the presentation of $u^{n}_{h}$ first. Multiplying $\zeta^{n}$ on both sides of Eq. \eqref{eqfullscheme} and summing $n$ from $1$ to $\infty$, we have
\begin{equation*}
	\sum_{n=1}^{\infty}\frac{u^{n}_{h}-u^{n-1}_{h}}{\tau}\zeta^{n}+\sum_{n=1}^{\infty}\sum_{i=0}^{n-1}d^{(1-\alpha)}_{i}A_{h}u^{n-i}_{h}\zeta^{n}=\sum_{n=1}^{\infty}P_{h}f(u^{n-1}_{h})\zeta^{n}+\sum_{n=1}^{\infty}P_{h}\xi^{H_{1},H_{2}}_{R,n}\zeta^{n}.
\end{equation*}
Combining the definition of $d^{(1-\alpha)}_{i}$, one has
\begin{equation*}
	\begin{aligned}
		\sum_{n=1}^{\infty}u^{n}_{h}\zeta^{n}=&(\delta_{\tau}(\zeta))^{\alpha-1}((\delta_{\tau}(\zeta))^{\alpha}+A_{h})^{-1}P_{h}\sum_{n=1}^{\infty}f(u^{n-1}_{h})\zeta^{n}\\
		&+(\delta_{\tau}(\zeta))^{\alpha-1}((\delta_{\tau}(\zeta))^{\alpha}+A_{h})^{-1}P_{h}\sum_{n=1}^{\infty}\xi^{H_{1},H_{2}}_{R,n}\zeta^{n}.
	\end{aligned}
\end{equation*}
Using Cauchy's integral formula and doing simple calculations lead to
\begin{equation}\label{eqfullschsol}
	\begin{aligned}
		u^{n}_{h}=&\frac{\tau}{2\pi\mathbf{i}}\int_{\Gamma_{\theta,\kappa}^{\tau}}e^{zt_{n}}(\delta_{\tau}(e^{-z\tau}))^{\alpha-1}((\delta_{\tau}(e^{-z\tau}))^{\alpha}+A_{h})^{-1}P_{h}\sum_{j=1}^{\infty}f(u^{j-1}_{h})e^{-zj\tau}dz\\
		&+\frac{\tau}{2\pi\mathbf{i}}\int_{\Gamma_{\theta,\kappa}^{\tau}}e^{zt_{n}}(\delta_{\tau}(e^{-z\tau}))^{\alpha-1}((\delta_{\tau}(e^{-z\tau}))^{\alpha}+A_{h})^{-1}P_{h}\sum_{j=1}^{\infty}\xi^{H_{1},H_{2}}_{R,j}e^{-zj\tau}dz,\\
	\end{aligned}
\end{equation}
where $\Gamma^\tau_{\theta,\kappa}=\{z\in \mathbb{C}:\kappa\leq |z|\leq\frac{\pi}{\tau\sin(\theta)},|\arg z|=\theta\}\cup\{z\in \mathbb{C}:|z|=\kappa,|\arg z|\leq\theta\}$. According to the definition of $\xi^{H_{1},H_{2}}_{R}$, we can obtain
\begin{equation*}
	\sum_{n=1}^{\infty}\xi^{H_{1},H_{2}}_{R,n}e^{-zt_{n}}=\frac{z}{e^{z\tau}-1}\tilde{\xi}^{H_{1},H_{2}}_{R}.
\end{equation*}
Introduce $\bar{F}(t)$ as
\begin{equation*}
	\bar{F}(t)=\left\{
	\begin{aligned}
		&0\qquad\qquad~~ t=t_{0},\\
		&f(u_{h}^{j-1})\qquad t\in(t_{j-1},t_{j}],
	\end{aligned}
	\right.
\end{equation*}
and $F(t)=f(u_{h}(t))$. In what  follows, we abbreviate $P_{h}F(t)$ and $P_{h}\bar{F}(t)$ as $F_{h}(t)$ and $\bar{F}_{h}(t)$.

Simple calculations lead to
\begin{equation*}
	\sum_{n=1}^{\infty}\bar{F}_{h}(t_{n})e^{-zt_{n}}=\frac{z}{e^{z\tau}-1}\tilde{\bar{F}}_{h}(z).
\end{equation*}
Thus there holds
\begin{equation*}
	\begin{aligned}
		u^{n}_{h}=&\int_{0}^{t_{n}}\bar{\mathcal{R}}_{h}(t_{n}-s)\bar{F}_{h}(s)ds\\
		&+\int_{0}^{t_{n}}\int_{D}\bar{\mathcal{G}}_{R,h}(t-s,x,y)\xi^{H_{1},H_{2}}(dy,s)ds,
	\end{aligned}
\end{equation*}
where
\begin{equation}\label{eqdefbGRh}
	\bar{\mathcal{G}}_{R,h}(t,x,y)=\sum_{k=1}^{\infty}\bar{\mathcal{G}}_{R,h,k}(t,x,y),\qquad \bar{\mathcal{G}}_{R,h,k}(t,x,y)=\bar{\mathcal{R}}_{h}(t)P_{h}\phi_{k}(x)\phi_{R,k}(y),
\end{equation}
and
\begin{equation}\label{eqdefbRh}
	\bar{\mathcal{R}}_{h}(t)=\frac{1}{2\pi\mathbf{i}}\int_{\Gamma_{\theta,\kappa}^{\tau}}e^{zt}(\delta_{\tau}(e^{-z\tau}))^{\alpha-1}((\delta_{\tau}(e^{-z\tau}))^{\alpha}+A_{h})^{-1}\frac{z\tau}{e^{z\tau}-1}dz.
\end{equation}
Let $\{\lambda_{k,h},\phi_{k,h}\}_{k=1}^{N}$ be the eigenvalues and eigenfunctions of operator $A_{h}$.
Thus $\bar{\mathcal{R}}_{h}(t)$ can also be written as, for $u_{h}\in X_{h}$,
\begin{equation*}
	\bar{\mathcal{R}}_{h}(t)u_{h}=\sum_{k=1}^{N}E_{k,h}(t)(u_{h},\phi_{k,h})\phi_{k,h},
\end{equation*}
where
\begin{equation}\label{eqdefEKH}
	E_{k,h}(t)=\frac{1}{2\pi\mathbf{i}}\int_{\Gamma_{\theta,\kappa}^{\tau}}e^{zt}(\delta_{\tau}(e^{-z\tau}))^{\alpha-1}((\delta_{\tau}(e^{-z\tau}))^{\alpha}+\lambda_{k,h})^{-1}\frac{z\tau}{e^{z\tau}-1}dz.
\end{equation}
Similar to the proof of Lemma \ref{lemEkest}, one can obtain
\begin{lemma}\label{lemEkhest}
	Let $E_{k,h}(t)$ be defined in \eqref{eqdefEKH}. For $k=1,2,\ldots,N$, $\sigma\in[0,1]$, and $\gamma\geq 0$, if $\gamma+\sigma\alpha<\frac{1}{2}$, then there exists a uniform constant $C$ such that
	\begin{equation}
		\|\lambda^{\sigma}_{k,h}E_{k,h}(t)\|_{H^{\gamma}((0,T))}=\|{}_{0}\partial^{\gamma}_{t}\lambda^{\sigma}_{k,h}E_{k,h}(t)\|_{L^{2}((0,T))}\leq C.
	\end{equation}
\end{lemma}

In the rest of paper, we take $\kappa\leq \frac{\pi}{t_{n}\sin(\theta)}$. Then we provide the temporal error
estimate.
\begin{theorem}\label{thm:tem}
	Let $\bar{u}_{h}$ and $u_{h}^{n}$ be the solution of \eqref{eqbaru0} and \eqref{eqfullscheme}, respectively. Let $f(u)$ satisfy  the assumptions \eqref{eqnonassump}. Assume $2H_{2}+(H_{1}-1)\alpha>0$. Then we have
	\begin{equation*}
		\mathbb{E}\|\bar{u}_{h}(t_{n})-u^{n}_{h}\|_{L^{2}(D)}^{2}\leq C\tau^{2H_{2}+(H_{1}-1)\alpha-\epsilon}.
	\end{equation*}
\end{theorem}
\begin{proof}
	Subtracting \eqref{eqfullschsol} from \eqref{eqsemisol}, we have
	\begin{equation*}
		\begin{aligned}
			&\mathbb{E}\|\bar{u}_{h}(t_{n})-u^{n}_{h}\|_{L^{2}(D)}^{2}\\
			\leq&C\mathbb{E}\left \|\int_{0}^{t_{n}}\mathcal{R}_{h}(t_{n}-s)F(s)-\bar{\mathcal{R}}_{h}(t_{n}-s)\bar{F}_{h}(s)ds\right \|_{L^{2}(D)}^{2}\\
			&+C\mathbb{E}\Bigg \|\int_{0}^{t_{n}}\int_{D}\mathcal{G}_{R,h}(t_{n}-s,x,y)\xi^{H_{1},H_{2}}(dy,ds)\\
			&\qquad\qquad-\int_{0}^{t_{n}}\int_{D}\bar{\mathcal{G}}_{R,h}(t_{n}-s,x,y)\xi^{H_{1},H_{2}}(dy,s)ds\Bigg \|_{L^{2}(D)}^{2}\\
			\leq&C\mathbb{E}\left \|\int_{0}^{t_{n}}\mathcal{R}_{h}(t_{n}-s)F(s)-\bar{\mathcal{R}}_{h}(t_{n}-s)\bar{F}_{h}(s)ds\right \|_{L^{2}(D)}^{2}\\
			&+C\mathbb{E}\left \|\int_{0}^{t_{n}}\int_{D}(\mathcal{G}_{R,h}(t_{n}-s,x,y)-\bar{\mathcal{G}}_{R,h}(t_{n}-s,x,y))\xi^{H_{1},H_{2}}(dy,ds)\right \|_{L^{2}(D)}^{2}\\
			&+C\mathbb{E}\Bigg \|\int_{0}^{t_{n}}\int_{D}\bar{\mathcal{G}}_{R,h}(t_{n}-s,x,y)(\xi^{H_{1},H_{2}}(dy,ds)-\xi^{H_{1},H_{2}}(dy,s)ds)\Bigg \|_{L^{2}(D)}^{2}\\
			\leq&\vartheta_{1}+\vartheta_{2}+\vartheta_{3}.
		\end{aligned}
	\end{equation*}	
	For $\vartheta_{1}$, one can split it into three parts, i.e.,
	\begin{equation*}
		\begin{aligned}
			\vartheta_{1}\leq& C\mathbb{E}\left \|\sum_{i=1}^{n}\int_{t_{i-1}}^{t_{i}}\mathcal{R}_{h}(t_{n}-s)(F_{h}(s)-F_{h}(t_{i-1}))ds\right \|_{L^{2}(D)}^{2}\\
			&+C\mathbb{E}\left \|\sum_{i=1}^{n}\int_{t_{i-1}}^{t_{i}}(\mathcal{R}_{h}(t_{n}-s)-\bar{\mathcal{R}}_{h}(t_{n}-s))F_{h}(t_{i-1})ds\right \|_{L^{2}(D)}^{2}\\
			& +C\mathbb{E}\left \|\sum_{i=1}^{n}\int_{t_{i-1}}^{t_{i}}\bar{\mathcal{R}}_{h}(t_{n}-s)(F_{h}(t_{i-1})-\bar{F}_{h}(s))ds\right \|_{L^{2}(D)}^{2}\\
			\leq& \vartheta_{1,1}+\vartheta_{1,2}+\vartheta_{1,3}.
		\end{aligned}
	\end{equation*}
	Using Theorem \ref{thmsobobaru}, one has
	\begin{equation*}
		\vartheta_{1,1}\leq C\tau^{2H_{2}+(H_{1}-1)\alpha-\epsilon}.
	\end{equation*}
	We first consider the estimate of $\|\mathcal{R}_{h}(t_{n}-s)-\bar{\mathcal{R}}_{h}(t_{n}-s)\|$. According to the definitions of $\mathcal{R}_{h}(t_{n}-s)$ and $\bar{\mathcal{R}}_{h}(t_{n}-s)$, there holds
	\begin{equation*}
		\begin{aligned}
			&\|\mathcal{R}_{h}(t_{n}-s)-\bar{\mathcal{R}}_{h}(t_{n}-s)\|\\
			\leq&C\left \|\int_{\Gamma_{\theta,\kappa}}e^{z(t_{n}-s)}z^{\alpha-1}(z^{\alpha}+A_{h})^{-1}dz\right .\\
			&\left .-\int_{\Gamma_{\theta,\kappa}^{\tau}}e^{z(t_{n}-s)}(\delta_{\tau}(e^{-z\tau}))^{\alpha-1}((\delta_{\tau}(e^{-z\tau}))^{\alpha}+A_{h})^{-1}\frac{z}{e^{z\tau}-1}dz\right \|\\
			\leq&C\left \|\int_{\Gamma_{\theta,\kappa}\backslash\Gamma_{\theta,\kappa}^{\tau}}e^{z(t_{n}-s)}z^{\alpha-1}(z^{\alpha}+A_{h})^{-1}dz\right \|\\
			&+C\bigg \|\int_{\Gamma_{\theta,\kappa}^{\tau}}e^{z(t_{n}-s)}\\
			&\qquad \cdot\left (z^{\alpha-1}(z^{\alpha}+A_{h})^{-1}-(\delta_{\tau}(e^{-z\tau}))^{\alpha-1}((\delta_{\tau}(e^{-z\tau}))^{\alpha}+A_{h})^{-1}\frac{z}{e^{z\tau}-1}\right )dz\bigg \|.\\
		\end{aligned}
	\end{equation*}
	By the fact $|z-\delta_{\tau}(e^{-z\tau})|\leq C\tau|z|^{2}$ \cite{Gunzburger.2018ScrotdfstPstaswn}, there is
	\begin{equation*}
		\left \|z^{\alpha-1}(z^{\alpha}+A_{h})^{-1}-(\delta_{\tau}(e^{-z\tau}))^{\alpha-1}((\delta_{\tau}(e^{-z\tau}))^{\alpha}+A_{h})^{-1}\frac{z}{e^{z\tau}-1}\right \|\leq C\tau,
	\end{equation*}
	which yields
	\begin{equation*}
		\begin{aligned}
			&\|\mathcal{R}_{h}(t_{n}-s)-\bar{\mathcal{R}}_{h}(t_{n}-s)\|\\ &\qquad\qquad \leq C\left(\tau^{1-\epsilon}\int_{\Gamma_{\theta,\kappa}\backslash\Gamma_{\theta,\kappa}^{\tau}}|e^{z(t_{n}-s)}||z|^{\epsilon}|dz|+\tau\int_{\Gamma_{\theta,\kappa}}|e^{z(t_{n}-s)}||dz|\right ).
		\end{aligned}
	\end{equation*}
	Thus combining the assumptions \eqref{eqnonassump} and Theorem \ref{thmsobobaru} and using Cauchy-Schwarz inequality, one can obtain
	\begin{equation*}
		\begin{aligned}
			\vartheta_{1,2}\leq &C\tau^{2-2\epsilon}\int_{0}^{t_{n}}(t_{n}-s)^{1-\epsilon}\left (\int_{\Gamma_{\theta,\kappa}\backslash\Gamma_{\theta,\kappa}^{\tau}}|e^{z(t_{n}-s)}||z|^{-\epsilon}|dz|\right )^{2}\mathbb{E}\|u(s)\|_{L^2(D)}^{2}ds\\
			&+C\tau^{2}\int_{0}^{t_{n}}(t_{n}-s)^{1-\epsilon}\left (\int_{\Gamma_{\theta,\kappa}^{\tau}}|e^{z(t_{n}-s)}||dz|\right )^{2}\mathbb{E}\|u(s)\|_{L^2(D)}^{2}ds\\
			\leq &C\tau^{2-2\epsilon}\int_{0}^{t_{n}}(t_{n}-s)^{\epsilon-1}ds\\
			&+C\tau^{2}\int_{0}^{t_{n}}(t_{n}-s)^{1-\epsilon}\int_{\Gamma_{\theta,\kappa}^{\tau}}|e^{2z(t_{n}-s)}||z|^{1-2\epsilon}|dz|\int_{\Gamma_{\theta,\kappa}^{\tau}}|z|^{2\epsilon-1}|dz|ds\\
			\leq &C\tau^{2-2\epsilon}.
		\end{aligned}
	\end{equation*}
	As for $\vartheta_{1,3}$, we arrive at
	\begin{equation*}
		\begin{aligned}
			\vartheta_{1,3}\leq C\tau\sum_{i=1}^{n-1}\|\bar{u}_{h}(t_{i})-u_{h}^{i}\|_{L^2(D)}^{2}.
		\end{aligned}
	\end{equation*}
	The fact $\|A_{h}^{s/2}P_{h}A^{-s/2}\|\leq C$ for $s\in[0,1]$ \cite{Thomee.2006GFEMfPP} and Theorem \ref{thmisometry} lead to
	\begin{equation*}
		\begin{aligned}
			\vartheta_{2}\leq&C\sum_{k=1}^{\infty}\int_{0}^{t_{n}}\left \|{}_{0}\partial^{\frac{1-2H_{2}}{2}}_{t}(\mathcal{R}_{h}(s)-\bar{\mathcal{R}}_{h}(s))P_{h}\phi_{k}(y)\right \|^{2}_{L^2(D)}\|\phi_{k}\|_{H^{\frac{1-2H_{1}}{2}}_{0}(D)}^{2}ds\\
			\leq&C\sum_{k=1}^{\infty}\lambda_{k}^{-\frac{1}{2}-\epsilon}\int_{0}^{t_{n}}\left \|\lambda_{k}^{\frac{1-H_{1}}{2}+\frac{\epsilon}{2}}{}_{0}\partial^{\frac{1-2H_{2}}{2}}_{t}(\mathcal{R}_{h}(s)-\bar{\mathcal{R}}_{h}(s))P_{h}\phi_{k}\right \|^{2}_{L^2(D)}ds\\
			\leq&C\sum_{k=1}^{\infty}\lambda_{k}^{-\frac{1}{2}-\epsilon}\int_{0}^{t_{n}}\left \|{}_{0}\partial^{\frac{1-2H_{2}}{2}}_{t}(\mathcal{R}_{h}(s)-\bar{\mathcal{R}}_{h}(s))P_{h}A^{\frac{1-H_{1}}{2}+\frac{\epsilon}{2}}\phi_{k}\right \|^{2}_{L^2(D)}ds\\
			\leq&C\int_{0}^{t_{n}}\left \|{}_{0}\partial^{\frac{1-2H_{2}}{2}}_{t}(\mathcal{R}_{h}(s)-\bar{\mathcal{R}}_{h}(s))A^{\frac{1-H_{1}}{2}+\frac{\epsilon}{2}}_{h}\right \|^{2}ds.
		\end{aligned}
	\end{equation*}
	Simple calculations imply
	\begin{equation*}
		\begin{aligned}
			&\left \|{}_{0}\partial^{\frac{1-2H_{2}}{2}}_{t}(\mathcal{R}_{h}(s)-\bar{\mathcal{R}}_{h}(s))A^{\frac{1-H_{1}}{2}+\frac{\epsilon}{2}}_{h}\right \|\\
			\leq&C\left\|\int_{\Gamma_{\theta,\kappa}\backslash\Gamma_{\theta,\kappa}^{\tau}}e^{zs}z^{\frac{1-2H_{2}}{2}}z^{\alpha-1}(z^{\alpha}+A_{h})^{-1}A_{h}^{\frac{1-H_{1}}{2}+\frac{\epsilon}{2}}dz\right\|\\
			&+C\left\|\int_{\Gamma_{\theta,\kappa}^{\tau}}e^{zs}z^{\frac{1-2H_{2}}{2}}(z^{\alpha-1}(z^{\alpha}+A_{h})^{-1}\right .\\
			&\left .\qquad-(\delta_{\tau}(e^{-z\tau}))^{\alpha-1}((\delta_{\tau}(e^{-z\tau}))^{\alpha}+A_{h})^{-1})A_{h}^{\frac{1-H_{1}}{2}+\frac{\epsilon}{2}}dz\right\|\\
			\leq&C\int_{\Gamma_{\theta,\kappa}\backslash\Gamma_{\theta,\kappa}^{\tau}}|e^{zs}||z|^{\frac{1-2H_{2}}{2}+(\frac{1-H_{1}}{2}+\frac{\epsilon}{2})\alpha-1}|dz|\\
			&+C\tau\int_{\Gamma_{\theta,\kappa}^{\tau}}|e^{zs}||z|^{\frac{1-2H_{2}}{2}+(\frac{1-H_{1}}{2}+\frac{\epsilon}{2})\alpha}|dz|,
		\end{aligned}
	\end{equation*}
	which yields
	\begin{equation*}
		\begin{aligned}
			\vartheta_{2}\leq&C\int_{0}^{t_{n}}\left(\int_{\Gamma_{\theta,\kappa}\backslash\Gamma_{\theta,\kappa}^{\tau}}|e^{zs}||z|^{\frac{1-2H_{2}}{2}+(\frac{1-H_{1}}{2}+\frac{\epsilon}{2})\alpha-1}|dz|\right )^{2}ds\\
			&+C\tau^{2}\int_{0}^{t_{n}}\left(\int_{\Gamma_{\theta,\kappa}^{\tau}}|e^{zs}||z|^{\frac{1-2H_{2}}{2}+(\frac{1-H_{1}}{2}+\frac{\epsilon}{2})\alpha}|dz|\right )^{2}ds\\
			\leq&C\int_{0}^{t_{n}}\int_{\Gamma_{\theta,\kappa}\backslash\Gamma_{\theta,\kappa}^{\tau}}|e^{2zs}||z|^{-\epsilon}|dz|\int_{\Gamma_{\theta,\kappa}\backslash\Gamma_{\theta,\kappa}^{\tau}}|z|^{1-2H_{2}+(1-H_{1}+\epsilon)\alpha-2+\epsilon}ds\\
			&+C\tau^{2}\int_{0}^{t_{n}}\int_{\Gamma_{\theta,\kappa}^{\tau}}|e^{2zs}|z|^{-\epsilon}|dz|\int_{\Gamma_{\theta,\kappa}^{\tau}}|z|^{1-2H_{2}+(1-H_{1}+\frac{\epsilon}{2})\alpha+\epsilon}|dz|ds\\
			\leq& C\tau^{2H_{2}+(H_{1}-1)\alpha-\epsilon_{0}}.
		\end{aligned}
	\end{equation*}
	As for $\uppercase\expandafter{\romannumeral3}$, we have
	\begin{equation*}
		\begin{aligned}
			\vartheta_{3}\leq&C\mathbb{E}\Bigg \|\sum_{k=1}^{\infty}\int_{0}^{t_{n}}\int_{D}\left (\bar{\mathcal{R}}_{h}(t_{n}-s)-\frac{1}{\tau}\sum_{i=1}^{n}\chi_{(t_{i-1},t_{i}]}(s)\int_{t_{i-1}}^{t_{i}}\bar{\mathcal{R}}_{h}(t_{n}-r)dr\right )\\
			&\qquad\qquad\qquad\cdot P_{h}\phi_{k}(x)\phi_{k}(y)\xi^{H_{1},H_{2}}(dy,ds))\Bigg \|_{L^{2}(D)}^{2}\\
			\leq&C\sum_{k=1}^{\infty}\int_{0}^{t_{n}}\left \|{}_{0}\partial^{\frac{1-2H_{2}}{2}}_{s}\left (\bar{\mathcal{R}}_{h}(s)-\frac{1}{\tau}\sum_{i=1}^{n}\chi_{[t_{i-1},t_{i})}(s)\int_{t_{i-1}}^{t_{i}}\bar{\mathcal{R}}_{h}(r)dr\right )P_{h}\phi_{k}(x)\right \|_{L^2(D)}^{2}\\
			&\qquad\qquad\qquad\cdot\|\phi_{k}(y)\|_{H^{\frac{1-2H_{1}}{2}}(D)}^{2}ds\\
			\leq&C\sum_{k=1}^{\infty}\lambda_{k}^{-\frac{1}{2}-2\epsilon}\\
			&\cdot\int_{0}^{t_{n}}\left \|{}_{0}\partial^{\frac{1-2H_{2}}{2}}_{s}\left (\bar{\mathcal{R}}_{h}(s)-\frac{1}{\tau}\sum_{i=1}^{n}\chi_{[t_{i-1},t_{i})}(s)\int_{t_{i-1}}^{t_{i}}\bar{\mathcal{R}}_{h}(r)dr\right )P_{h}\lambda_{k}^{\frac{1-H_{1}}{2}+\epsilon}\phi_{k}(x)\right \|_{L^2(D)}^{2}ds\\
			\leq&C\int_{0}^{t_{n}}\left \|{}_{0}\partial^{\frac{1-2H_{2}}{2}}_{s}\left (\bar{\mathcal{R}}_{h}(s)-\frac{1}{\tau}\sum_{i=1}^{n}\chi_{[t_{i-1},t_{i})}(s)\int_{t_{i-1}}^{t_{i}}\bar{\mathcal{R}}_{h}(r)dr\right )A_{h}^{\frac{1-H_{1}}{2}+\epsilon}\right \|_{X_{h}\rightarrow X_{h}}^{2}ds.
		\end{aligned}
	\end{equation*}
	Introduce $v_{h}=\sum_{k=1}^{N}a_{k}\phi_{k,h}$ with $\|v_{h}\|_{X_{h}}^{2}=(v_{h},v_{h})=\sum_{k=1}^{N}a_{k}^{2}$. Thus
	\begin{equation*}
		\begin{aligned}
			&\left \|{}_{0}\partial^{\frac{1-2H_{2}}{2}}_{s}\left (\bar{\mathcal{R}}_{h}(s)-\frac{1}{\tau}\sum_{i=1}^{n}\chi_{[t_{i-1},t_{i})}(s)\int_{t_{i-1}}^{t_{i}}\bar{\mathcal{R}}_{h}(r)dr\right )A_{h}^{\frac{1-H_{1}}{2}+\epsilon}\right \|_{X_{h}\rightarrow X_{h}}^{2}\\
			\leq &\sup_{v_{h}\in X_{h},\|v_{h}\|_{X_{h}}=1}\left \|{}_{0}\partial^{\frac{1-2H_{2}}{2}}_{s}\left (\bar{\mathcal{R}}_{h}(s)-\frac{1}{\tau}\sum_{i=1}^{n}\chi_{[t_{i-1},t_{i})}(s)\int_{t_{i-1}}^{t_{i}}\bar{\mathcal{R}}_{h}(r)dr\right )A_{h}^{\frac{1-H_{1}}{2}+\epsilon}v_{h}\right \|_{X_{h}}^{2}\\
			\leq &C\sup_{v_{h}\in X_{h},\|v_{h}\|_{X_{h}}=1}\sum_{k=1}^{N}a_{k}^{2}\left |{}_{0}\partial^{\frac{1-2H_{2}}{2}}_{s}\left (E_{k,h}(s)-\frac{1}{\tau}\sum_{i=1}^{n}\chi_{[t_{i-1},t_{i})}(s)\int_{t_{i-1}}^{t_{i}}E_{k,h}(r)dr\right )\lambda_{k,h}^{\frac{1-H_{1}}{2}+\epsilon}\right |^{2}.
		\end{aligned}
	\end{equation*}
	Combining Lemma \ref{lemEkhest} and the similar arguments of Theorem \ref{thmreglarerr}, we obtain
	\begin{equation*}
		\vartheta_{3}\leq C\tau^{2H_{2}+(H_{1}-1)\alpha-\epsilon}.
	\end{equation*}
	After gathering the above estimates and using the discrete Gr\"onwall inequality \cite{Thomee.2006GFEMfPP}, the desired result can be reached.
\end{proof}
\section{Numerical experiments}\label{sec5}
Here we present some numerical examples to show the effectiveness of the numerical methods and confirm the theoretical results. We take $m$ trajectories $\{\omega_{j}\}_{j=1}^{m}$ to calculate the solution. Since the exact solutions are unknown in the numerical experiments, we measure the convergence rates by
\begin{equation*}
	{\rm Rate}=\frac{\ln(e_{h}/e_{h/2})}{\ln(2)},\quad {\rm Rate}=\frac{\ln(e_{\tau}/e_{\tau/2})}{\ln(2)},
\end{equation*}
where
\begin{equation*}
	\begin{aligned}
		&e_{h}=\left (\frac{1}{m}\sum_{i=1}^{m}\|u^{M}_{h}(\omega_{i})-u^{M}_{h/2}(\omega_{i})\|^{2}_{L^{2}(D)}\right )^{1/2},\\
		&e_{\tau}=\left (\frac{1}{m}\sum_{i=1}^{m}\|u_{\tau}(\omega_{i})-u_{\tau/2}(\omega_{i})\|^{2}_{L^{2}(D)}\right )^{1/2},
	\end{aligned}
\end{equation*}
with $u^{M}_{h}$ and $u_{\tau}$ being the solutions at time $T=t_{M}$ with mesh size $h$ and time step size $\tau$, respectively.

\begin{example}
	In this example, to show the temporal convergence, we take $m=200$, $\beta=1$, $T=0.5$, $l=0.5$, $f(u)=\sin(u)$, and $h=l/512$ to solve \eqref{eqretosol}. The numerical results with different $\alpha$, $H_{1}$, and $H_{2}$ are given in Table \ref{tab:tem}, where
	the numbers in the bracket in the last column denote the theoretical rates predicted by Theorem \ref{thm:tem}. As it can be seen, there is a good agreement between the numerical convergence rates and the predicted ones.
	\begin{table}[htbp]
		\centering
		\caption{Temporal errors and convergence rates.}
		\begin{tabular}{c|ccccl}
			\hline\noalign{\smallskip}
			$(\alpha,H_1,H_2)\backslash T/\tau$ & 16    & 32    & 64    & 128   & Rate \\
			\noalign{\smallskip}\hline\noalign{\smallskip}
			(0.3,0.2,0.2) & 4.002E-05 & 4.191E-05 & 4.155E-05 & 3.478E-05 & $\approx0.0675(=0.08)$ \\
			(0.3,0.3,0.5) & 1.012E-05 & 8.111E-06 & 5.976E-06 & 4.188E-06 & $\approx0.4243(=0.395)$ \\
			(0.5,0.2,0.3) & 6.808E-05 & 5.872E-05 & 5.805E-05 & 5.069E-05 & $\approx0.1419(=0.1)$ \\
			(0.5,0.4,0.3) & 5.808E-05 & 5.563E-05 & 4.244E-05 & 4.497E-05 & $\approx0.1230(=0.15)$ \\
			(0.7,0.4,0.5) & 3.886E-05 & 3.354E-05 & 2.822E-05 & 2.126E-05 & $\approx0.2901(=0.29)$ \\
			(0.7,0.5,0.2) & 1.339E-04 & 1.351E-04 & 1.482E-04 & 1.268E-04 & $\approx0.0264(=0.025) $\\
			\noalign{\smallskip}\hline
		\end{tabular}%
		\label{tab:tem}%
	\end{table}%
\end{example}
\begin{example}
	To show the spatial convergence rates, we consider the numerical solution of \eqref{eqretosol} with $\beta=10$ and $f(u)=\frac{1}{50}\sin(u)$.
	Here, we take $m=100$, $T=0.01$, $l=0.1$, and $\tau=T/2048$. We show the corresponding errors and convergence rates with different $\alpha$, $H_{1}$, and $H_{2}$ in Table \ref{tab:spa}. All the numerical convergence rates well agree with the predicted ones stated in Theorem \ref{thm:spa}.
	\begin{table}[htbp]
		\centering
		\caption{Spatial errors and convergence rates.}
		\begin{tabular}{c|ccccl}
			\hline
			\noalign{\smallskip}
			$(\alpha,H_1,H_2)\backslash l/h$ & 8     & 16    & 32    & 64    &  Rate\\
			\noalign{\smallskip}\hline\noalign{\smallskip}
			(0.3,0.2,0.5) & 1.385E-01 & 7.186E-02 & 5.079E-02 & 3.173E-02 & $\approx0.7087(=0.7)$ \\
			(0.3,0.5,0.5) & 3.188E-02 & 1.554E-02 & 7.991E-03 & 4.002E-03 & $\approx0.9979(=1)$ \\
			(0.5,0.2,0.3) & 9.038E-01 & 6.727E-01 & 4.861E-01 & 3.072E-01 & $\approx0.5190(=0.4)$ \\
			(0.5,0.5,0.4) & 1.533E-01 & 7.697E-02 & 4.082E-02 & 2.042E-02 & $\approx0.9695(=1)$ \\
			(0.7,0.4,0.4) & 4.977E-01 & 3.482E-01 & 2.396E-01 & 1.468E-01 & $\approx0.5871(=0.5429)$ \\
			(0.7,0.3,0.4) & 7.310E-01 & 5.083E-01 & 3.810E-01 & 2.546E-01 & $\approx0.5073(=0.4429)$ \\
			\noalign{\smallskip}\hline
		\end{tabular}%
		\label{tab:spa}%
	\end{table}%
\end{example}
\section{Conclusions}\label{sec6}

Anomalous diffusions are ubiquitous in the nature world. The Brownian motion subordinated by inverse $\alpha$-stable L\'evy process can effectively model the subdiffusion. In this paper, we introduce its Fokker-Planck equation with nonlinear source term and external fractional noise, and put all our efforts on the numerical methods of the equation. That is, we approximate the stochastic nonlinear fractional diffusion equation driven by the fractional Brownian sheet noise with Hurst parameters $H_{1},H_{2}\in(0,\frac{1}{2}]$. After providing the regularity of the solution and regularizing the rough noise by Wong-Zakai approximation, we build the fully discrete scheme by backward Euler convolution quadrature and finite element methods. Moreover, the complete error analyses are also developed. Finally, the numerical experiments validate the effectiveness of the designed algorithm.

%


%
%

\bibliographystyle{spmpsci}
\bibliography{ref2}

\end{document}